\documentclass[12pt,english]{amsart}
\usepackage[T1]{fontenc}
\usepackage{geometry}
\usepackage{amssymb}

\makeatletter

\newcommand{\noun}[1]{\textsc{#1}}

 \theoremstyle{plain}
\newtheorem{thm}{Theorem}[section]
  \theoremstyle{plain}
  \newtheorem{lem}[thm]{Lemma}
  \theoremstyle{plain}
  \newtheorem{prop}[thm]{Proposition}
  \theoremstyle{remark}
  \newtheorem{rem}[thm]{Remark}
  \theoremstyle{remark}
  \newtheorem*{rem*}{Remark}


\makeatother

\begin{document}

\title{On a question of Landis and Oleinik}

\author{Tu Nguyen}

\begin{abstract}
Let $P=\partial_{t}+\partial_{i}(a^{ij}\partial_{j})$ be a backward
parabolic operator. It is shown that under certain conditions on $\left\{ a^{ij}\right\} $,
if $u$ satisfies $\left|Pu\right|\leq C(\left|u\right|+\left|\nabla u\right|)$,
$\left|u(x,t)\right|\lesssim e^{C\left|x\right|^{2}}$ in $\mathbb{R}^{n}\times[0,T]$
and $\left|u(x,0)\right|\lesssim e^{-M\left|x\right|^{2}}$ for all
$M>0$, then $u$ vanishes identically in $\mathbb{R}^{n}\times[0,T]$.
\end{abstract}
\maketitle

\section{Introduction}

Let $P$ be a backward parabolic operator on $\mathbb{R}^{n}$,\[
Pu=\partial_{t}u+\mbox{div}(A\nabla u)\]
where $A(x,t)=(a^{ij}(x,t))_{i,j=1}^{n}$ is a real, symmetric matrix
such that for some $\lambda>0$, \begin{equation}
\lambda\left|\xi\right|^{2}\leq a^{ij}(x,t)\xi_{i}\xi_{j}\leq\lambda^{-1}\left|\xi\right|^{2}\mbox{ for all }\xi\in\mathbb{R}^{n}.\label{eq:ellipticity}\end{equation}
It was conjectured by Landis and Oleinik \cite{MR0402268} that if
$Pu=b(x,t)\cdot\nabla u+a(x,t)u$ in $\mathbb{R}^{n}\times[0,T]$
and $\left|u(x,0)\right|\lesssim e^{-\left|x\right|^{2+\epsilon}},\,\forall x\in\mathbb{R}^{n}$
then $u\equiv0$ in $\mathbb{R}^{n}\times[0,T]$, provided $A,b$
and $c$ satisfy appropriate conditions at infinity.

Escauriaza, Kenig, Ponce and Vega \cite{MR2231129} showed that this
is true when $P$ is the backward heat operator (i.e. $A(x,t)\equiv\mbox{Id}$)
and $b$ and $c$ are bounded. They also obtained a similar result
when the domain is $\mathbb{R}_{+}^{n}\times[0,T]$. The aim of this
paper is to extend these results to parabolic operators with variable
coefficients. 

\begin{thm}
Suppose that $\{a^{ij}\}$ satisfy the ellipticity condition (\ref{eq:ellipticity}),
and for some $\epsilon>0$\begin{eqnarray*}
 &  & \left|\nabla_{x}a^{ij}(x,t)\right|\lesssim\left\langle x\right\rangle ^{-1-\epsilon},\,\,\,\left|\partial_{t}a^{ij}(x,t)\right|\lesssim1,\\
 &  & \left|a^{ij}(x,t)-a^{ij}(x,s)\right|\lesssim\left\langle x\right\rangle ^{-1}\left|t-s\right|^{1/2},\,\,\,\,\,\,\,\,\,\,\,\,\,\,\,\forall x\in\mathbb{R}^{n};\,\,\, t,s\in[0,T].\end{eqnarray*}
Assume that $u$ satisfies the inequalities \[
\left|Pu\right|\leq C(\left|u\right|+\left|\nabla u\right|)\,\,\,\,\,\mbox{ in }\mathbb{R}^{n}\times[0,T]\]
and \[
\left|u(x,t)\right|\lesssim e^{C\left|x\right|^{2}}\,\,\,\,\,\,\,\,\forall(x,t)\in\mathbb{R}^{n}\times[0,T],\]
 for some $C>0$. Then
\begin{enumerate}
\item [1. ]If $\left|u(x,0)\right|\lesssim e^{-M\left|x\right|^{2}}$ for
all $M>0$, then $u\equiv0$.
\item [2. ]If $u(x,0)\not\equiv0$ then there exists $M>0$ such that if
$\left|x\right|>M$, \[
\int_{B(x,1)}\left|u(y,0)\right|^{2}dy\gtrsim e^{-M\left|x\right|^{2}\log\left|x\right|}\,\,\,\,\,\,\,\mbox{and\,\,\,\,\,\,\,}\int_{B(x,\left|x\right|/2)}\left|u(y,0)\right|^{2}dy\gtrsim e^{-M\left|x\right|^{2}}.\]

\end{enumerate}
\end{thm}
~

We also obtain a similar result where the domain is a half-space.

\begin{thm}
Let $\mathbb{R}_{+}^{n}=\left\{ x\in\mathbb{R}^{n}:x_{1}>0\right\} $.
Suppose that $\{a^{ij}\}$ satisfy the ellipticity condition (\ref{eq:ellipticity}),
and for some $\epsilon>0$\begin{eqnarray*}
 &  & \left|\nabla_{x}a^{ij}(x,t)\right|\lesssim\left\langle x\right\rangle ^{-1-\epsilon},\,\,\,\left|\partial_{t}a^{ij}(x,t)\right|\lesssim1,\\
 &  & \left|a^{ij}(x,t)-a^{ij}(x,s)\right|\lesssim\left\langle x\right\rangle ^{-1}\left|t-s\right|^{1/2},\,\,\,\,\,\,\,\,\,\,\,\,\,\,\,\forall x\in\mathbb{R}_{+}^{n};\,\,\, t,s\in[0,T].\end{eqnarray*}
Assume that $u$ satisfies the inequalities \[
\left|Pu\right|\leq C(\left|u\right|+\left|\nabla u\right|)\,\,\,\,\,\mbox{ in }\mathbb{R}_{+}^{n}\times[0,T]\]
and \[
\left|u(x,t)\right|\lesssim e^{C\left|x\right|^{2}}\,\,\,\,\,\,\,\,\forall(x,t)\in\mathbb{R}_{+}^{n}\times[0,T],\]
for some $C>0$. Then
\begin{enumerate}
\item [1. ]If $\left|u(x,0)\right|\lesssim e^{C\left|x\right|^{2}-Mx_{1}^{2}}$
for all $M>0$, then $u\equiv0$.
\item [2. ]If $u(\cdot,0)\not\equiv0$ then there exists $M>0$ such that
if $R>M$, \[
\int_{B(Re_{1},1)}\left|u(y,0)\right|^{2}dy\gtrsim e^{-MR^{2}\log R}\,\,\,\,\,\,\,\mbox{and\,\,\,\,\,\,\,}\int_{B(Re_{1},R/2)}\left|u(y,0)\right|^{2}dy\gtrsim e^{-MR^{2}}.\]

\end{enumerate}
\end{thm}
The proof in \cite{MR2231129} for the heat operator used a Carleman
inequality together with a scaling argument to show $u(x,0)$ has
a doubling property which implies that \[
\int_{B(x,\left|x\right|/2)}\left|u(y,0)\right|^{2}dy\gtrsim e^{-M\left|x\right|^{2}}\]
for some $M>0$. This argument breaks down in the variable coefficients
case, as it requires a uniform bound on $\left\Vert \nabla a_{x_{0},R}^{ij}\right\Vert _{L^{\infty}}$
where $a_{x_{0},R}^{ij}(x,t)=a^{ij}(x_{0}+Rx,R^{2}t)$ and $R>\left|x_{0}\right|$.

To prove the first part of Theorem 1.1, we first show that if $\left|u(x,0)\right|\lesssim e^{-M\left|x\right|^{2}}$
for all $M>0$ then there exists $T_{0}\in[0,T]$ such that for any
$M\geq0$, \[
\int_{0}^{T_{0}}\int_{B(x,\left|x\right|/2)}u^{2}(y,t)dydt\lesssim e^{-M\left|x\right|^{2}}\mbox{ \,\,\,\,\,\, if }\left|x\right|\geq R_{M}.\]
Then we show that if $u(\cdot,0)\not\equiv0$, for any $T_{0}\in[0,T]$,
the following lower bound holds \[
\int_{0}^{T_{0}}\int_{B(x,\left|x\right|/2)}u^{2}(y,t)dydt\gtrsim e^{-C_{2}\left|x\right|^{2}}\]
where $C_{2}=C_{2}(T_{0},u)\geq0$. (a similar bound for the Schr\"odinger
equation was proved in \cite{MR2273975}.) Thus, we must have $u(\cdot,0)\equiv0$,
which then implies $u\equiv0$ in $\mathbb{R}^{n}\times[0,T]$. The
proof of Theorem 1.2 follows the same argument, using anisotropic
Carleman inequalities instead, as now $u$ decays in the direction
of $x_{1}$ only.

We would like to mention a unique continuation result of \cite{MR1971939,MR2001174}.
Let $u$ be a solution of the inequality $\left|Pu\right|\leq M(\left|u\right|+\left|\nabla u\right|)$
in $B(0,1)\times[0,T]$ which vanishes to infinite order at $0$,
i.e. $\left|u(x,0)\right|\leq C_{k}\left|x\right|^{k}$ for all $k\geq0$,
$\forall x\in B(0,1)$. Then $u(\cdot,0)\equiv0$ in $B(0,1)$. We
have benefited from the Carleman inequalities and ideas contained
in these papers, and also from those of \cite{MR040,MR1979722,MR2005639,MR2273975,MR2231129}.

Details of the proofs of Theorem 1.1 and 1.2 are in section 2 and
3, respectively. The proofs of the Carleman inequalities used in section
2 and 3 will be gathered in section 4, together with other auxiliary
lemmas.

\textbf{Acknowledgement.} I would like to thank my thesis advisor
Carlos Kenig for suggesting the problem, and for his invaluable guidance
and support.

\section{Proof of theorem 1.1}

We first remark that by considering $P_{r}=\partial_{t}+\mbox{div}\left(A_{r}\nabla\right)$
where $a_{r}^{ij}(x,t)=a^{ij}(rx,r^{2}t)$ and $u_{r}(x,t)=u(rx,r^{2}t)$
for suitably small $r>0$, we can assume that the constant $C$ in
the hypothesis of Theorem 1.1 is as small as we like, say $C\leq\lambda^{5}/100$,
and that $\left|\partial_{t}a^{ij}(x,t)\right|\leq C$. Furthermore,
we can take $T=1$.

\subsection{Upper bound}

In this section we will adapt the arguments of \cite{MR2005639} to
show that under the hypothesis of Theorem 1.1, there exists $T_{0}>0$
such that for any $M>0$, if $\left|x\right|\geq R_{M}>0$ \begin{equation}
\left|u(x,t)\right|+\left|\nabla u(x,t)\right|\lesssim e^{-M\left|x\right|^{2}}\mbox{ for all }t\in[0,T_{0}].\label{eq: upperbound}\end{equation}
First, we prove the weaker bound \begin{equation}
\left|u(x,t)\right|+\left|\nabla u(x,t)\right|\lesssim e^{-M\left|x\right|^{2}}\,\,\,\,\,\mbox{if }0\leq t\lesssim M^{-1}.\label{eq:weak upperbound}\end{equation}
(Note that the time interval of this weaker bound shrinks as $M\rightarrow\infty$.)
Then we combine this bound with $M=2$ and another Carleman inequality
to obtain (\ref{eq: upperbound}).

\subsubsection{First step}

We will use the following Carleman inequality of \cite{MR1971939}.

\begin{lem}
\label{1st Carleman}Suppose $a^{ij}(0,0)=\delta_{ij}$ and $\left|a^{ij}(x,t)-a^{ij}(y,s)\right|\leq L\left(\left|x-y\right|+\left|s-t\right|^{1/2}\right)$.
Then there is a constant $N=N(n,\lambda,L)>0$ such that for any $\alpha\geq2$
there is a positive function $\sigma:(0,\frac{4}{\alpha})\rightarrow\mathbb{R}_{+}$
satisfying \[
N^{-1}\leq\frac{\sigma(t)}{t}\leq1\]
so that if $v\in C_{c}^{\infty}(\mathbb{R}^{n}\times[0,\frac{2}{\alpha}))$
and $0<a<1/\alpha$, then 

\begin{eqnarray*}
 &  & \int_{\mathbb{R}^{n+1}}(\alpha^{2}v^{2}+\alpha\sigma_{a}\left|\nabla v\right|^{2})\sigma_{a}^{-\alpha}G_{a}dxdt\leq N\int_{\mathbb{R}^{n+1}}\sigma_{a}^{1-\alpha}\left|Pv\right|^{2}G_{a}dxdt\\
 &  & +\sigma(a)^{-\alpha}\left[-\frac{a}{N}\int_{\mathbb{R}^{n}}\left|\nabla v(x,0)\right|^{2}G_{a}(x,0)dx+\alpha N\int_{\mathbb{R}^{n}}v^{2}(x,0)G_{a}(x,0)dx\right]\\
 &  & +\alpha^{\alpha}N^{\alpha}\sup_{t\geq0}\int_{\mathbb{R}^{n}}(v^{2}+\left|\nabla v\right|^{2})dx\end{eqnarray*}
Here $G_{a}(x,t)=(t+a)^{-n/2}e^{-\left|x\right|^{2}/4(t+a)}$ and
$\sigma_{a}(t)=\sigma(t+a)$.
\end{lem}
Since the hypothesis of the lemma requires $a^{ij}(0,0)=\delta_{ij}$,
we first need to make a change of variable. Let $x_{0}\in\mathbb{R}^{n}$
with $\left|x_{0}\right|\gtrsim1$. Let $S=A(x_{0},0)^{1/2}$, $z_{0}=S^{-1}x_{0}$,
$\widetilde{u}(x,t)=u(Sx,t)$ and $\widetilde{A}(x,t)=S^{-1}A(Sx,t)S^{-1}$.
Then $\widetilde{A}(z_{0},0)=\mbox{Id}$ and \[
\partial_{t}\widetilde{u}+\mbox{div}(\widetilde{A}\nabla\widetilde{u})|_{(x,t)}=\partial_{t}u+\mbox{div}(A\nabla u)|_{(Sx,t)}.\]
Let $\widetilde{u}_{R}$ be a rescale of $\widetilde{u}$ centered
at $z_{0}$, $\widetilde{u}_{R}(x,t)=\widetilde{u}(z_{0}+Rx,R^{2}t)$
where $R=\lambda\left|x_{0}\right|/4$. Then $\widetilde{u}_{R}$
satisfies \[
\left|P_{R}\widetilde{u}_{R}\right|\leq R^{2}\left|\widetilde{u}_{R}\right|+R\left|\nabla\widetilde{u}_{R}\right|\]
where $P_{R}=\partial_{t}+\mbox{div}(\widetilde{A}_{R}\nabla)$, and
$\widetilde{A}_{R}(x,t)=\widetilde{A}(z_{0}+Rx,R^{2}t)$. 

From the hypothesis of Theorem 1.1 and our choice of $R$, it is easy
to see that \[
\left|\nabla\widetilde{a}_{R}^{ij}(x,t)\right|\lesssim1,\left|\widetilde{a}_{R}^{ij}(x,t)-\widetilde{a}_{R}^{ij}(x,s)\right|\lesssim\left|t-s\right|^{1/2}\]
in $B(0,2)\times[0,1/R^{2})$. Furthermore, $\widetilde{A}_{R}(0,0)=\widetilde{A}(z_{0},0)=\mbox{Id}$.
Thus, we can apply Lemma \ref{1st Carleman} to $P_{R}$ and $v=\widetilde{u}_{R}\psi(x)\varphi(t)$,
where $\chi_{[0,1/\alpha]}\leq\varphi\leq\chi_{[0,2/\alpha)}$ and
$\chi_{B(0,1)}\leq\psi\leq\chi_{B(0,2)}$ are bump functions, $\alpha\geq2R^{2}$
is a positive constant to be chosen. Let $E=B(0,2)\times[0,2/\alpha)\backslash B(0,1)\times[0,1/\alpha]$
then \[
\left|P_{R}v\right|\leq R^{2}\left|v\right|+R\left|\nabla v\right|+\alpha(\left|\widetilde{u}_{R}\right|+\left|\nabla\widetilde{u}_{R}\right|)\chi_{E}.\]
Hence, by Lemma \ref{1st Carleman}, \begin{eqnarray*}
\int_{\mathbb{R}^{n+1}}(\alpha^{2}v^{2}+\alpha\sigma_{a}\left|\nabla v\right|^{2})\sigma_{a}^{-\alpha}G_{a}dxdt & \lesssim & N\int_{\mathbb{R}^{n+1}}\sigma_{a}^{1-\alpha}(R^{2}\left|v\right|+R\left|\nabla v\right|)^{2}G_{a}dxdt+\\
 &  & N\int_{E}\sigma_{a}^{1-\alpha}\alpha^{2}(\left|\widetilde{u}_{R}\right|+\left|\nabla\widetilde{u}_{R}\right|)^{2}G_{a}dxdt+\\
 &  & \alpha^{\alpha}N^{\alpha}\sup_{t\geq0}\int_{\mathbb{R}^{n}}(v^{2}+\left|\nabla v\right|^{2})dx+\\
 &  & \alpha N\sigma(a)^{-\alpha}\int_{\mathbb{R}^{n}}v^{2}(x,0)G_{a}dx.\end{eqnarray*}
If $\alpha\geq2NR^{2}$ then the first term on the right hand side
can be absorbed by the left hand side. Also, $\sigma_{a}(t)^{-\alpha}G_{a}(x,t)\leq N^{\alpha}\alpha^{\alpha+\frac{n}{2}}$
in $E$, and $\left|\widetilde{u}_{R}\right|+\left|\nabla\widetilde{u}_{R}\right|\lesssim e^{CR^{2}}$
by hypothesis on $u$ and Lemma \ref{basic L8 bound} in the Appendix.
Thus, we obtain\begin{eqnarray*}
\int_{\mathbb{R}^{n+1}}(\alpha^{2}v^{2}+\alpha\sigma_{a}\left|\nabla v\right|^{2})\sigma_{a}^{-\alpha}G_{a}dxdt & \lesssim & N^{\alpha}\alpha^{\alpha+\frac{n}{2}}e^{2CR^{2}}+\alpha N\sigma(a)^{-\alpha}\int_{\mathbb{R}^{n}}v^{2}(x,0)G_{a}dx.\end{eqnarray*}

Let $\rho=\frac{1}{Ne}$, and $a=\frac{\rho^{2}}{2\alpha}$. Then
\[
\sigma_{a}(t)^{-\alpha+1}G_{a}(x,t)\geq\alpha^{\alpha+\frac{n}{2}-1}N^{2\alpha+n-2}\mbox{ \,\,\,\,\,\,\,\,\,\,\,\, in }B(0,2\rho)\times[0,\frac{\rho^{2}}{2\alpha}]\]
and \[
\sigma(a)^{-\alpha}G_{a}(x,0)\leq N^{\alpha}a^{-\alpha-\frac{n}{2}}=(2\alpha e^{2})^{\alpha+\frac{n}{2}}N^{3\alpha+n}\,\,\,\,\mbox{for all }x.\]
Hence, \begin{eqnarray*}
 &  & \alpha^{\alpha+\frac{n}{2}}N^{2\alpha+n-2}\int_{B(0,2\rho)\times[0,\frac{\rho^{2}}{2\alpha}]}(v^{2}+\left|\nabla v\right|^{2})dxdt\lesssim N^{\alpha}\alpha^{\alpha+\frac{n}{2}}e^{2CR^{2}}\\
 &  & \,\,\,\,\,\,\,\,\,\,\,\,\,\,\,\,\,\,\,\,\,\,\,\,\,\,\,\,\,\,\,\,\,+\alpha^{\alpha+\frac{n}{2}+1}(2e)^{2\alpha+n}N^{3\alpha+n+1}\int_{B(0,2)}v^{2}(x,0)dx\end{eqnarray*}
or \[
\int_{B(0,2\rho)\times[0,\frac{\rho^{2}}{2\alpha}]}(v^{2}+\left|\nabla v\right|^{2})dxdt\lesssim N^{2-\alpha-n}e^{2CR^{2}}+\alpha(2e)^{2\alpha+n}N^{\alpha+3}\int_{B(0,2)}v^{2}(x,0)dx.\]
We now choose $\alpha=MR^{2}$ then the first term in the right hand
side is bounded by $e^{-MR^{2}}$. The second term is also bounded
by $e^{-MR^{2}}$ by the decay hypothesis on $u(\cdot,0)$. Thus,
for any $M>2N$, \[
\int_{B(0,2\rho)\times[0,\frac{\rho^{2}}{2MR^{2}}]}(v^{2}+\left|\nabla v\right|^{2})dxdt\lesssim e^{-MR^{2}}.\]
By Lemma \ref{basic L8 bound}, this implies \[
\left|v\right|+\left|\nabla v\right|\lesssim e^{-MR^{2}}\,\,\,\,\,\,\mbox{in }B(0,\rho)\times[0,\frac{\rho^{2}}{4MR^{2}}].\]
Undoing the change of variable, we get \[
\left|u(x_{0},t)\right|+\left|\nabla u(x_{0},t)\right|\lesssim e^{-MR^{2}}\,\,\,\,\,\,\mbox{if }0\leq t\leq\frac{\rho^{2}}{4M}.\]
This proves (\ref{eq:weak upperbound}).

\subsubsection{Second step.}

\begin{lem}
\label{lem:upper 1} Let $\epsilon$  be the constant in the hypothesis
of Theorem 1.1. Let \[
G(x,t)=\exp\left(c(T-t)\left|x\right|+\left|x\right|^{2}\right)\]
 where $0\leq c\leq R^{1+\epsilon/8}$. Then for any $v\in C_{c}^{\infty}(\left\{ R\leq\left|x\right|\leq R^{1+\epsilon/8}\right\} \times[0,T])$,
the following inequality holds \begin{eqnarray*}
 &  & \frac{\lambda^{2}}{4}\int_{\mathbb{R}^{n+1}}v^{2}Gdxdt+\frac{\lambda^{2}}{4}\int_{\mathbb{R}^{n+1}}\left|\nabla v\right|^{2}Gdxdt\leq\int_{\mathbb{R}^{n+1}}\left|Pv\right|^{2}Gdxdt\\
 &  & +\lambda^{-1}\int_{\mathbb{R}^{n}}\left|\nabla v(x,T)\right|^{2}G(x,T)dx+R^{2+\epsilon/4}\int_{\mathbb{R}^{n}}v^{2}(x,0)G(x,0)dx,\end{eqnarray*}
 provided $R\gtrsim1$.
\end{lem}
This Carleman inequality is an extension of a Carleman inequality
in \cite{MR2005639} to the case of variable coefficients. As $\{a^{ij}\}$
are no longer constants, it is necessary to put a restriction on the
support of $v$ (compared with \cite{MR2005639}.) We will prove this
inequality in the Appendix. We now deduce (\ref{eq: upperbound})
from (\ref{eq:weak upperbound}) and Lemma \ref{lem:upper 1}. 

\begin{prop}
\label{pro:upper} Suppose that $u$ is as in the hypothesis of Theorem
1.1, and $\left|u(x,0)\right|\lesssim e^{-M\left|x\right|^{2}}$ for
all $M>0$. Let $T=\rho^{2}/8N$, where $\rho$ and $N$ are as above.
Then for all $M>0$,\[
\left|u(x,t)\right|+\left|\nabla u(x,t)\right|\lesssim e^{-M\left|x\right|^{2}}\]
for all $t\in[0,T/4]$.
\end{prop}
\begin{proof}
Fix $M>0$. Let \[
v(x,t)=u(x,t)\theta(x)\]
where \[
\theta(x)=\begin{cases}
\begin{array}{cc}
0 & \mbox{ if }\left|x\right|<R-1\mbox{ or }\left|x\right|>MR+1\\
1 & \mbox{ if }R<\left|x\right|<MR\end{array}\end{cases}\]
Since \[
Pv=\theta Pu+2\left\langle A\nabla u,\nabla\theta\right\rangle +u\Delta\theta,\]
it follows that\begin{eqnarray*}
\left|Pv\right| & \leq & C\theta\left(\left|u\right|+\left|\nabla u\right|\right)+2\lambda^{-1}\left|\nabla u\right|\left|\nabla\theta\right|+\left|u\Delta\theta\right|\\
 & \leq & C\left(\left|v\right|+\left|\nabla v\right|\right)+\left|u\right|\left(C\left|\nabla\theta\right|+\left|\Delta\theta\right|\right)+2\lambda^{-1}\left|\nabla u\right|\left|\nabla\theta\right|\\
 & \leq & C\left(\left|v\right|+\left|\nabla v\right|\right)+C'\left(\left|u\right|+\left|\nabla u\right|\right)\chi_{E}\end{eqnarray*}
where $E=\left(\left\{ R-1<\left|x\right|<R\right\} \cup\{MR<\left|x\right|<MR+1\}\right)\times[0,T]$,
$C'\leq4\lambda^{-1}$. 

Choose $c=MR/T$. Then for large $R$, $c\leq R^{1+\epsilon/8}$ and
$\mbox{supp}v\subset\left\{ R\leq\left|x\right|\leq R^{1+\epsilon/8}\right\} $.
Thus, we can apply the previous lemma to $v$ to obtain\begin{eqnarray*}
\int_{0}^{T}\int_{\mathbb{R}^{n}}(v^{2}+\left|\nabla v\right|^{2})Gdxdt & \lesssim & \int_{\mathbb{R}^{n}}\left|\nabla v(x,T)\right|^{2}G(x,T)dx+R^{2+\epsilon/4}\int_{\mathbb{R}^{n}}v^{2}(x,0)G(x,0)dx\\
 &  & +\int_{E}\left(\left|u\right|^{2}+\left|\nabla u\right|^{2}\right)Gdxdt.\end{eqnarray*}
In the previous subsection, we have shown that 

\[
\left|u(x,t)\right|+\left|\nabla u(x,t)\right|\lesssim e^{-2N\left|x\right|^{2}}\mbox{ for all }t\in[0,T].\]
Hence, as $G(x,T)=e^{\left|x\right|^{2}}$, we have \[
\int_{\mathbb{R}^{n}}\left|\nabla v(x,T)\right|^{2}Gdx\lesssim1.\]
Since $\left|u(x,0)\right|\lesssim e^{-2M\left|x\right|^{2}}$ and
$G(x,0)\leq e^{(M+1)\left|x\right|^{2}}$ if $\left|x\right|\geq R$
by our choice of $c$, it follows that \[
R^{2+\epsilon/4}\int_{\mathbb{R}^{n}}v^{2}(x,0)Gdx\lesssim1.\]
 In $\{MR<\left|x\right|<MR+1\}$, $G(x,t)\leq e^{2\left|x\right|^{2}}$,
hence, \[
\int_{0}^{T}\int_{MR<\left|x\right|<MR+1}\left(\left|u\right|^{2}+\left|\nabla u\right|^{2}\right)Gdxdt\lesssim R^{n-1}.\]
In $\left\{ R-1<\left|x\right|<R\right\} $, $G(x,t)\leq e^{(M+1)R^{2}}$,
so

\[
\int_{0}^{T}\int_{R-1<\left|x\right|<R}\left(\left|u\right|^{2}+\left|\nabla u\right|^{2}\right)Gdxdt\lesssim e^{(M+2)R^{2}}.\]
Thus,\[
\int_{0}^{T}\int_{\mathbb{R}^{n}}(v^{2}+\left|\nabla v\right|^{2})Gdxdt\lesssim e^{(M+2)R^{2}}.\]
As $G(x,t)\geq e^{4MR^{2}}$ in $\{6R\leq\left|x\right|\leq7R\}\times[0,T/2]$,
this implies \begin{eqnarray*}
\int_{0}^{T/2}\int_{6R\leq\left|x\right|\leq7R}\left(\left|u\right|^{2}+\left|\nabla u\right|^{2}\right)dxdt & \lesssim & e^{-MR^{2}},\end{eqnarray*}
provided $R\geq R_{M}$. This and Lemma \ref{basic L8 bound} prove
the proposition.
\end{proof}

\subsection{Lower Bound}

In this subsection, assuming $u(\cdot,0)\not\equiv0$, we will show
that the following lower bound holds for any $T\leq1$, \begin{equation}
\int_{0}^{T}\int_{R<\left|x\right|<2R}u^{2}(x,0)dx\gtrsim e^{-C_{2}R^{2}}\label{eq:lower}\end{equation}
To prove this, we first adapt arguments of \cite{MR2198840} and \cite{MR040}
to show that there exists $s>0,$ such that for small $t$, we have
\begin{equation}
\int_{R<\left|x\right|<2R}u^{2}(x,t)dx\gtrsim e^{-R^{s}}.\label{eq:weak lower}\end{equation}
Then we use this bound together with a bootstrap argument to obtain
\eqref{eq:lower}.

\subsubsection{First step}

Since $u(\cdot,0)\not\equiv0$, we can suppose that, \[
\int_{B(e_{1},\rho\lambda/4)}u^{2}(x,0)dx\ne0.\]
Here $\rho$ is a positive constant to be chosen. By using Lemma \ref{Lemma},
and multiplying $u$ by a constant if necessary, we can assume that
\begin{equation}
\int_{B(e_{1},\rho\lambda/2)}u^{2}(x,t)dx\geq L\label{eq:trivia}\end{equation}
if $t$ is small enough. Here $L$ is a large constant to be chosen.

We will use the doubling property of $u(\cdot,0)$ proved by Escauriaza,
Fern\'andez and Vessella. We present their arguments here in the
form that we need. Let $x_{0}=\left|x_{0}\right|e_{1}$, and $v$
be as in section 2.1. As before, if $\alpha\geq2NR^{2}$ the following
inequality holds\begin{eqnarray}
 &  & \int_{\mathbb{R}^{n+1}}(\alpha^{2}v^{2}+\alpha\sigma_{a}\left|\nabla v\right|^{2})\sigma_{a}^{-\alpha}G_{a}dxdt\leq N^{\alpha}\alpha^{\alpha+\frac{n}{2}}e^{CR^{2}}+\nonumber \\
 &  & \sigma(a)^{-\alpha}\left[-\frac{a}{N}\int_{\mathbb{R}^{n}}\left|\nabla v(x,0)\right|^{2}G_{a}dx+\alpha N\int_{\mathbb{R}^{n}}v^{2}(x,0)G_{a}dx\right].\label{eq:first 1}\end{eqnarray}
Let $\rho=\frac{1}{Ne}$ and $0<a\leq\rho^{2}/2\alpha$. Then, \begin{eqnarray}
\alpha^{2}\int_{\mathbb{R}^{n+1}}v^{2}\sigma_{a}^{-\alpha}G_{a} & \geq & \alpha^{2}\int_{0}^{\rho^{2}/\alpha}dt\int_{B(0,2\rho)}(t+a)^{-\alpha-\frac{n}{2}}e^{-\rho^{2}/(t+a)}v^{2}(x,t)dx\nonumber \\
 & \geq & N_{\rho}\alpha^{2}\int_{a}^{a+\rho^{2}/\alpha}s^{-\alpha-\frac{n}{2}}e^{-\rho^{2}/s}ds\int_{B(0,\rho)}v^{2}(x,0)dx\label{eq:first 2}\\
 & \geq & N_{\rho}\alpha^{2}\int_{\rho^{2}/2\alpha}^{\rho^{2}/\alpha}s^{-\alpha-\frac{n}{2}}e^{-\rho^{2}/s}ds\int_{B(0,\rho)}v^{2}(x,0)dx\nonumber \\
 & \geq & \frac{N_{\rho}\alpha^{\alpha+\frac{n}{2}+1}N^{2\alpha}}{2}\int_{B(0,\rho)}v^{2}(x,0)dx.\nonumber \end{eqnarray}
(we have used Lemma \ref{Lemma} in the second inequality. $N_{\rho}$
is the constant appears in that lemma.) Here, $\alpha$ has to satisfy
\[
\rho^{2}/\alpha\leq N_{\rho}^{-1}\min\left\{ R^{-2},1/\log\left(\frac{N_{\rho}\int_{B(0,1)\times[0,R^{-2}]}v^{2}(x,t)dxdt}{\int_{B(0,\rho)}v^{2}(x,0)dx}\right)\right\} \]
As $\left|v(x,t)\right|\lesssim e^{CR^{2}}$, we can take \[
\alpha=\rho^{2}N_{\rho}\left(2R^{2}+\log\frac{N_{\rho}}{\int_{B(0,\rho)}v^{2}(x,0)dx}\right).\]
For this value of $\alpha$, \[
\frac{N_{\rho}\alpha^{\alpha+\frac{n}{2}+1}N^{2\alpha}}{2}\int_{B(0,\rho)}v^{2}(x,0)dx\geq N^{\alpha}\alpha^{\alpha+\frac{n}{2}}e^{CR^{2}}.\]
This together with (\ref{eq:first 1}) and (\ref{eq:first 2}) show
that \[
-\frac{a}{N}\int_{\mathbb{R}^{n}}\left|\nabla v(x,0)\right|^{2}G_{a}(x,0)dx+\alpha N\int_{\mathbb{R}^{n}}v^{2}(x,0)G_{a}(x,0)dx\geq0\]
or,\[
2a\int_{\mathbb{R}^{n}}\left|\nabla v(x,0)\right|^{2}G_{a}(x,0)dx+\frac{n}{2}\int_{\mathbb{R}^{n}}v^{2}(x,0)G_{a}(x,0)dx\leq4\alpha N^{2}\int_{\mathbb{R}^{n}}v^{2}(x,0)G_{a}(x,0)dx\]
for all $a\leq\rho^{2}/2\alpha$. 

By Lemma \ref{Lmm doubl}, this implies that \[
\int_{B(0,2r)}v^{2}(x,0)\leq e^{128\alpha N^{2}}\int_{B(0,r)}v^{2}(x,0)\]
for all $0\leq r\leq1/2$. It follows that there exists positive constant
$C_{1}$ and $C_{2}$ such that if $r\leq\rho/2$, \begin{equation}
\left(\int_{B(0,\rho)}v^{2}(x,0)\right)^{1+C_{1}\log\frac{\rho}{r}}\leq e^{C_{2}R^{2}\log\frac{\rho}{r}}\int_{B(0,r)}v^{2}(x,0).\label{eq:true doubling}\end{equation}
Taking $r=\rho\lambda^{2}/2$, we see that there are constants $J$
and $K$ so that \[
\left(\int_{B(0,\rho)}v^{2}(x,0)\right)^{K}\leq e^{JR^{2}/2}\int_{B(0,\rho\lambda^{2}/2)}v^{2}(x,0).\]
This implies, after undoing the changes of variable, \begin{equation}
\left(\int_{B(x_{0},\rho\lambda R)}u^{2}(x,0)\right)^{K}\leq\lambda^{-K}e^{JR^{2}}\int_{B(x_{0},\rho\lambda R/2)}u^{2}(x,0).\label{eq:doubling}\end{equation}
We now use a chain-of-balls argument similar to that of \cite{MR040}.
Let $x_{k+1}=(1-\frac{\rho\lambda^{2}}{8})x_{k}$ for $k=0,1,2,\ldots$.
Then by (\ref{eq:doubling}), 

\[
\left(\int_{B(x_{k+1},\rho\lambda^{2}\left|x_{k+1}\right|/8)}u^{2}(x,0)\right)^{K}\leq\lambda^{-K}e^{J\left|x_{k}\right|^{2}}\int_{B(x_{k},\rho\lambda^{2}\left|x_{k}\right|/8)}u^{2}(x,0)\,\,\,\,\,\,\,\, k=0,1,\ldots\]
Let $m=[\log\left|x_{0}\right|/\log\frac{8}{8-\rho\lambda^{2}}]$
then $\left|x_{m}\right|\sim1$, hence \[
\int_{B(x_{m},\rho\lambda^{2}\left|x_{m}\right|/8)}u^{2}(x,0)\geq\lambda^{K}e^{-J\left|x_{m}\right|^{2}}\left(\int_{B(e_{1},\rho\lambda^{2}/8)}u^{2}(x,0)\right)^{K}\geq1.\]
(we have used (\ref{eq:trivia}) in the last inequality.) It follows
that \[
\int_{B(x_{0},\rho\lambda^{2}\left|x_{0}\right|/8)}u^{2}(x,0)\geq\lambda^{\frac{K^{m+1}-K}{K-1}}e^{-\frac{J(K^{m}-1)}{K-1}\left|x_{0}\right|^{2}},\]
which, by the choice of $m$, implies \[
\int_{B(x_{0},\rho\lambda\left|x_{0}\right|/2)}u^{2}(x,0)\geq e^{-C_{s}\left|x_{0}\right|^{s}}\]
for some positive constants $s$ and $C_{s}$. The same inequality
holds for $u(\cdot,t)$ if $t$ is small so that (\ref{eq:trivia})
holds.

\subsubsection{Second step.}

We now use (\ref{eq:weak lower}) and another Carleman inequality
to prove (\ref{eq:lower}). Let $\psi\in C_{c}^{\infty}(0,T)$ be
a positive bump function satisfying \[
\psi(t)=\begin{cases}
\begin{array}{cc}
0 & \mbox{ if }t\in[0,\frac{T}{8}]\cup[\frac{7T}{8},T]\\
4 & t\in[\frac{T}{4},\frac{3T}{4}]\end{array}.\end{cases}\]
Let $\delta\in(1,1+\epsilon/2)$, where $\epsilon$ is the constant
in the hypothesis of Theorem 1.1. Let \[
S_{R,T}:=\left\{ (x,t):R^{1/\delta}\leq\left|x\right|\leq R,T/8\leq t\leq7T/8\right\} .\]

\begin{lem}
\label{lem:lower 1}Let $G(x,t)=e^{\varphi(x,t)}$ where $\varphi(x,t)=E_{1}R(T-t)\left|x\right|+E_{2}\left|x-R\psi(t)e_{1}\right|^{2}$.
Here $E_{1}\gtrsim T^{-2},E_{2}\gtrsim1$ are constants that may depend
on $R$, but $E_{1}/E_{2}\geq100/T$ is a fixed constant. Then if
$R\geq R_{0}=R_{0}(E_{1}/E_{2},T,\lambda)$, \begin{eqnarray*}
 &  & E_{1}^{3}R^{2}\int_{\mathbb{R}_{+}^{n+1}}v^{2}Gdxdt+E_{2}\int_{\mathbb{R}_{+}^{n+1}}\left|\nabla v\right|^{2}Gdxdt\lesssim\int_{\mathbb{R}_{+}^{n+1}}\left|Pv\right|^{2}Gdxdt,\end{eqnarray*}
for any $v\in C_{c}^{\infty}(S_{R,T})$. The implicit constant depends
only on $T$ and $\lambda$.
\end{lem}
We give a proof of this lemma in the Appendix. Note that in contrast
to Lemma \ref{lem:upper 1}, here the main term in $\varphi$ is $E_{1}R(T-t)\left|x\right|$,
as $E_{1}\gg E_{2}$. The use of the shift $x-R\psi(t)e_{1}$ originates
in a Carleman inequality for Schr\"odinger equations proved in \cite{MR2273975}
(see their Lemma 3.1) 

The next proposition, a corollary of this lemma, is the basis of our
bootstrap argument.

\begin{prop}
\label{pro:iteration}Let $u$ be as in Theorem 1.1. Suppose that
for some $s\geq2$, there exist $C_{s}>0$ such that\[
\int_{T/4}^{3T/4}\int_{R\leq\left|x\right|\leq2R}(u^{2}+\left|\nabla u\right|^{2})dxdt\gtrsim\exp(-C_{s}R^{s})\]
 for all $R\geq C_{s}$. Let $s_{1}=\max\left\{ 2,\frac{s-1}{\delta}+1\right\} $,
where $1<\delta<1+\frac{\epsilon}{2}$. Then there is $C_{s_{1}}>0$
such that 

\[
\int_{0}^{T}\int_{R-1\leq\left|x\right|\leq R}(u^{2}+\left|\nabla u\right|^{2})dxdt\gtrsim\exp(-C_{s_{1}}R^{s_{1}})\]
for all $R\geq C_{s_{1}}$.
\end{prop}
\begin{proof}
Let $v(x,t)=u(x,t)\theta(x,t)$ where $\theta(x,t)=\theta_{1}(x)\theta_{2}(x-R\psi(t)e_{1})$,
with $\psi$ defined as above, and \[
\theta_{1}(x)=\begin{cases}
\begin{array}{cc}
0 & \mbox{ if }\left|x\right|<R^{1/\delta}\mbox{ or }\left|x\right|>cR\\
1 & \mbox{ if }R^{1/\delta}+1\leq\left|x\right|\leq cR-1\end{array}\end{cases}\]
\[
\theta_{2}(x)=\begin{cases}
\begin{array}{cc}
0 & \mbox{ if }\left|x\right|<2R\\
1 & \mbox{ if }\left|x\right|>3R,\end{array}\end{cases}\]
with $c=2^{-11}$. Clearly, $\mbox{supp}(v)\subset S_{R,T}$.

We have\begin{eqnarray*}
\left|Pv\right| & \leq & C\left(\left|v\right|+\left|\nabla v\right|\right)+\left|u\right|\left(C\left|\nabla\theta\right|+\left|\partial_{t}\theta\right|+\left|\Delta\theta\right|\right)+2\lambda^{-1}\left|\nabla u\right|\left|\nabla\theta\right|\\
 & \leq & C\left(\left|v\right|+\left|\nabla v\right|\right)+C'\left(\left|u\right|+\left|\nabla u\right|\right)\chi_{E}\end{eqnarray*}
 where $E=\mbox{supp}\nabla\theta$ and $C'\leq\lambda^{4}/T$. 

Applying the previous Carleman inequality to $v$, we get\[
\int_{0}^{T}\int_{\mathbb{R}^{n}}(v^{2}+\left|\nabla v\right|^{2})Gdxdt\lesssim\int_{E}\left(\left|u\right|^{2}+\left|\nabla u\right|^{2}\right)Gdxdt.\]
Since\begin{eqnarray*}
\inf_{\begin{array}{cc}
16R^{1/\delta}\leq\left|x\right|\leq32R^{1/\delta}\\
T/4\leq t\leq3T/4\end{array}}\left\{ G(x,t)\right\}  & \geq & \exp\left(4E_{1}TR^{1+\frac{1}{\delta}}+E_{2}\left(4R-32R^{1/\delta}\right)^{2}\right),\end{eqnarray*}
if $E_{1}\geq2\cdot16^{s}C_{s}R^{\frac{s-1}{\delta}-1}$ and $E_{1}/E_{2}=256/T$
then \begin{eqnarray*}
\int_{0}^{T}\int_{\mathbb{R}^{n}}(\left|v\right|^{2}+\left|\nabla v\right|^{2})Gdxdt & \gtrsim & \exp\left(4E_{1}TR^{1+\frac{1}{\delta}}+E_{2}\left(4R-32R^{1/\delta}\right)^{2}-16^{s}C_{s}R^{s/\delta}\right)\\
 & \geq & \exp\left(3E_{1}TR^{1+\frac{1}{\delta}}+16E_{2}R^{2}\right)=:\Sigma\end{eqnarray*}
The set $E$ is contained in the union of $\{R^{1/\delta}\leq\left|x\right|\leq R^{1/\delta}+1\}$,
$\{2R\leq\left|x-R\psi(t)e_{1}\right|\leq3R\}\cap\{\left|x\right|\leq cR\}$
and $\{cR-1\leq\left|x\right|\leq cR\}$. In $\{R^{1/\delta}\leq\left|x\right|\leq R^{1/\delta}+1\}$,
\[
G(x,t)\leq\exp\left(2E_{1}TR^{1+\frac{1}{\delta}}+E_{2}\left(4R+2R^{1/\delta}\right)^{2}\right)\]
and $\left|u\right|+\left|\nabla u\right|\lesssim e^{CR^{2/\delta}}$,
hence\begin{eqnarray*}
 &  & \int_{0}^{T}\int_{R^{1/\delta}\leq\left|x\right|\leq2R^{1/\delta}}\left(\left|u\right|^{2}+\left|\nabla u\right|^{2}\right)Gdxdt\lesssim\\
 &  & \exp\left(2E_{1}TR^{1+\frac{1}{\delta}}+16E_{2}R^{2}+20E_{2}R^{1+\frac{1}{\delta}}+CR^{2/\delta}\right)\ll\Sigma/4.\end{eqnarray*}
In $\{2R\leq\left|x-R\psi(t)e_{1}\right|\leq3R\}\cap\{\left|x\right|\leq cR\}$,
\[
G(x,t)\leq\exp\left(c^{2}E_{1}TR^{2}+9E_{2}R^{2}\right)\leq\exp(10E_{2}R^{2})\]
hence \[
\int_{0}^{T}\int_{2R\leq\left|x-R\psi(t)e_{1}\right|\leq3R,\left|x\right|\leq cR}\left(\left|u\right|^{2}+\left|\nabla u\right|^{2}\right)Gdxdt\ll\Sigma/4.\]
Thus, we conclude that \[
\Sigma/4\leq\int_{0}^{T}\int_{cR-1\leq\left|x\right|\leq cR}\left(\left|u\right|^{2}+\left|\nabla u\right|^{2}\right)Gdxdt.\]
Since in $\{cR-1\leq\left|x\right|\leq cR\}$, $G\leq\exp(25E_{2}R^{2})$,
we obtain \[
\int_{0}^{T}\int_{cR-1\leq\left|x\right|\leq cR}(\left|u\right|^{2}+\left|\nabla u\right|^{2})dxdt\geq\exp(-9E_{2}R^{2}).\]
Recall that we need $E_{1}\geq2\cdot16^{s}C_{s}R^{\frac{s-1}{\delta}-1}$
and $E_{1}\gtrsim T^{-2}$. With the minimum choice $E_{1}\sim\max\{1,R^{\frac{s-1}{\delta}-1}\}$,
we obtain\[
\int_{0}^{T}\int_{cR-1\leq\left|x\right|\leq cR}(\left|u\right|^{2}+\left|\nabla u\right|^{2})dxdt\geq\exp(-C_{s_{1}}R^{s_{1}}).\]
for large $R$. The proposition follows from this.
\end{proof}
\begin{prop}
\label{pro:lower}Suppose $u$ satisfies the assumption of Theorem
1.1. If $u(\cdot,0)\not\equiv0$ then for any $T\leq1$, there exist
$C_{2}=C_{2}(T,u)>0$ such that 

\[
\int_{0}^{T}\int_{R-1\leq\left|x\right|\leq R}(u^{2}+\left|\nabla u\right|^{2})dxdt\gtrsim\exp(-C_{2}R^{2})\]
for all $R\geq C_{2}$.
\end{prop}
\begin{proof}
This is a consequence of repeatedly applying the previous proposition.
Let $s_{0}=s$ where $s$ is the exponent appeared in (\ref{eq:weak lower}),
and\[
s_{k+1}=2+\left(\frac{s_{k}-1}{\delta}-1\right)_{+},\,\,\,\,\, k=1,2,3,\ldots\]
It is simple to check that there is $k_{0}$ such that $s_{k}=2$
for all $k\geq k_{0}$. Clearly, we can assume that on $[0,T]$, (\ref{eq:weak lower})
holds. Let $a_{k}=T\left(\frac{1}{2}-2^{k-k_{0}-1}\right)$ and $b_{k}=T\left(\frac{1}{2}+2^{k-k_{0}-1}\right)$.
Since \[
\int_{a_{0}}^{b_{0}}\int_{R<\left|x\right|<2R}\left|u(x,t)\right|^{2}dxdt\gtrsim e^{-R^{s}},\]
the previous proposition (applied to the time interval $[a_{1},b_{1}]$)
shows that \[
\int_{a_{1}}^{b_{1}}\int_{R-1\leq\left|x\right|\leq R}(u^{2}+\left|\nabla u\right|^{2})dxdt\gtrsim\exp(-C_{s_{1}}R^{s_{1}})\,\,\,\,\,\,\,\mbox{if }R\geq C_{s_{1}}\]
for some positive $C_{s_{1}}$. Induction then shows that for any
$k$, there is $C_{s_{k}}>0$ such that \[
\int_{a_{k}}^{b_{k}}\int_{R-1\leq\left|x\right|\leq R}(u^{2}+\left|\nabla u\right|^{2})dxdt\gtrsim\exp(-C_{s_{k}}R^{s_{k}})\,\,\,\,\,\,\,\mbox{if }R\geq C_{s_{1}}.\]
In particular when $k=k_{0}$ we obtain \[
\int_{0}^{T}\int_{R-1\leq\left|x\right|\leq R}(u^{2}+\left|\nabla u\right|^{2})dxdt\gtrsim e^{-C_{2}R^{2}}.\]

\end{proof}

\subsection{Proof of Theorem 1.1}

\begin{proof}
1. Suppose otherwise $u\not\equiv0$. We can assume without loss of
generality that $u(\cdot,0)\not\equiv0$. (if not, we can translate
to a time $0<s<1$ such that $u(\cdot,s)\not\equiv0$. The bounds
$\left|u(x,s)\right|\lesssim e^{-M\left|x\right|^{2}}$ for all $M$,
follows from (\ref{eq: upperbound})). But then we are in position
to apply Proposition \ref{pro:lower}, and obtain a lower bound that
contradicts the upper bound of Proposition \ref{pro:upper}. Thus,
we must have $u\equiv0$.

2. Let $T=\rho^{2}/8N$. Inspecting the proof of Proposition \ref{pro:upper},
we see that to obtain the upper bound\[
\left|u(x,t)\right|+\left|\nabla u(x,t)\right|\lesssim e^{-M\left|x\right|^{2}}\mbox{ in }(B_{7R}\backslash B_{6R})\times[0,T/4],\]
for some $M\geq2N$, it suffices to have \[
\int_{B(x,1)}u^{2}(y,0)\leq e^{-2M\left|x\right|^{2}}\]
for all $x\in B_{2MR}\backslash B_{R/2}$. Hence, in order to avoid
contradiction with the lower bound (\ref{eq:lower}), we must have
\[
\sup_{x\in B_{2MR}\backslash B_{R/2}}\int_{B(x,1)}u^{2}(y,0)\geq e^{-4M^{2}R^{2}},\]
if $M\geq2\max\{N,C_{2}\}$ (here $C_{2}$ is the constant appears
in Proposition \ref{pro:lower}). This and (\ref{eq:doubling}) together
with a chain-of-balls argument shows that \[
\inf_{x\in B_{MR}\backslash B_{R}}\int_{B(x,\rho\lambda R)}u^{2}(y,0)\geq e^{-M_{1}R^{2}},\]
for some $M_{1}>0$. Combining this with the doubling inequality (\ref{eq:true doubling}),
we obtain

\[
\inf_{x\in B_{MR}\backslash B_{R}}\int_{B(x,1)}u^{2}(y,0)\geq e^{-M_{2}R^{2}\log R}.\]
 These estimates prove the second part of the theorem.
\end{proof}
\begin{rem}
As the cutoff functions used in the proof of Theorem 1.1 are radial,
the same results and proofs apply to solutions of $\left|Pu\right|\lesssim\left|u\right|+\left|\nabla u\right|$
in $\left(\mathbb{R}^{n}\backslash B_{R}\right)\times[0,1].$
\end{rem}

\section{Proof of theorem 1.2}

The proof of Theorem 1.2 is very similar to that of Theorem 1.1, using
anisotropic Carleman inequalities. We use the notation $x=(x_{1},x')$.

\subsection{Upper bound}

For the first step, the same argument as in section 2.1.1 shows that
for all $M>0$ \begin{equation}
\left|u(x,t)\right|+\left|\nabla u(x,t)\right|\lesssim e^{C\left|x\right|^{2}-Mx_{1}^{2}}\mbox{ for all }x\in\mathbb{R}_{+}^{n},\label{eq:thm 2 upper 1}\end{equation}
if $0\leq t\lesssim M^{-1}$(here $C$ is the constant in the statement
of Theorem 1.2. Now we can only rescale with $R\sim x_{1}$, resulting
in the weaker bound.)

For the second step, we will need the next lemma, which is inspired
by a Carleman inequality in \cite{MR1979722}. To ease notations,
we will assume that $a_{\infty}^{1j}=0$ for $j\ne1$ where $a_{\infty}^{ij}=\lim_{x\rightarrow\infty}a^{ij}(x,t)$.
Otherwise, we will need to replace $\varphi$ below by\[
\tilde{\varphi}(x,t)=\varphi(x_{1},Bx',t)\]
where $B$ is a positively definite, symmetric $(n-1)\times(n-1)$-matrix,
satisfying $\sum_{j\ne1}B^{ij}a_{\infty}^{1j}=0$ for all $i=2,3,\ldots,n$.
The reader can check that the conclusion of the lemma holds with such
a modification of $\varphi$. (we only use $a_{\infty}^{1j}=0$ to
control the term $I_{4}$ in the proof.)

\begin{lem}
Let $\epsilon$ be the constant in the hypothesis of Theorem 1.2.
Let $G(x,t)=e^{\varphi(x,t)}$ where \[
\varphi(x,t)=-\frac{\lambda\left|x'\right|^{2}}{8s}+\frac{c(S^{\alpha}-s^{\alpha})}{s^{\alpha}}x_{1}+bs.\]
Here $0\leq c\leq R^{1+\epsilon/8}$, $\alpha$ and $b\leq\alpha/4$
are large fixed constants, $s$ is the translated time variable $s=t+1$,
and $S=T+1$. Then for large $R$, for any $v\in C_{c}^{\infty}(\left\{ R\leq x_{1}\leq R^{1+\epsilon/8}\right\} \times[0,T])$,
\begin{eqnarray*}
 &  & \frac{1}{16}\int_{0}^{T}\int_{\mathbb{R}_{+}^{n}}(cRv^{2}+b\left|\nabla v\right|^{2})Gdxdt\leq\int_{0}^{T}\int_{\mathbb{R}_{+}^{n}}\left|Pv\right|^{2}Gdxdt+\int_{\mathbb{R}_{+}^{n}}\left\Vert \nabla v(x,T)\right\Vert ^{2}Gdx\\
 &  & +\int_{\mathbb{R}_{+}^{n}}(\left|x'\right|^{2}+R^{2+\epsilon})v^{2}(x,0)G(x,0)dx+\int_{\mathbb{R}_{+}^{n}}(\left|x'\right|^{2}+R^{2+\epsilon})v^{2}(x,T)G(x,T)dx\end{eqnarray*}

\end{lem}
We give a proof of this lemma in the Appendix. 

\begin{prop}
\label{pro:upper 1.2}Suppose that \begin{equation}
\left|u(x,t)\right|+\left|\nabla u(x,t)\right|\lesssim e^{C\left|x\right|^{2}-2^{\alpha}x_{1}^{2}}\mbox{ \,\,\,\,\,\,\, }\forall(x,t)\in\mathbb{R}_{+}^{n}\times[0,T],\label{eq: thm 2 upper}\end{equation}
Let $d=\frac{2^{\alpha+1}(T+2)}{\alpha T}$, where $\alpha$ is as
in the previous lemma. Then for any $M>0$ we have
\end{prop}
\[
\int_{0}^{T/2}\int_{dR<x_{1}<2dR,\left|x'\right|<R}\left(\left|u\right|^{2}+\left|\nabla u\right|^{2}\right)dxdt\lesssim e^{-MR^{2}}\]

\begin{proof}
Let \[
v(x,t)=u(x,t)\theta(x)\mbox{ where }\theta(x)=\theta_{1}(x)\theta_{2}(x),\]
 and \[
\theta_{1}(x)=\begin{cases}
\begin{array}{cc}
0 & \mbox{ if }x_{1}<R-1\mbox{ or }x_{1}>MR+1\\
1 & \mbox{ if }R<x_{1}<MR\end{array}\end{cases}\]
\[
\theta_{2}(x)=\begin{cases}
\begin{array}{cc}
1 & \mbox{ if }\left|x'\right|<r\\
0 & \mbox{ if }\left|x'\right|>r+1\end{array}\end{cases}\]
We will now apply the previous lemma with $c=MR$, to the function
$v$ and get\begin{eqnarray*}
 &  & \int_{0}^{T}\int_{\mathbb{R}_{+}^{n}}(v^{2}+\left|\nabla v\right|^{2})Gdxdt\lesssim\int_{E}\left(\left|u\right|^{2}+\left|\nabla u\right|^{2}\right)Gdxdt+\int_{\mathbb{R}_{+}^{n}}\left\Vert \nabla v(x,T)\right\Vert ^{2}Gdx\\
 &  & +\int_{\mathbb{R}_{+}^{n}}(\left|x'\right|^{2}+R^{2+\epsilon})v^{2}(x,0)G(x,0)dx+\int_{\mathbb{R}_{+}^{n}}(\left|x'\right|^{2}+R^{2+\epsilon})v^{2}(x,T)G(x,T)dx\end{eqnarray*}
where $E=\mbox{supp}\nabla\theta\times[0,T]$.

Using (\ref{eq: thm 2 upper}) and the decay of $u(\cdot,0)$, we
can check easily that the last three integrals in the right hand side
are bounded by $R^{2+\epsilon}$, and \[
\int_{\left\{ r<\left|x'\right|<r+1,R<x_{1}<MR\right\} \times[0,T]}\left(\left|u\right|^{2}+\left|\nabla u\right|^{2}\right)Gdxdt\rightarrow0\,\,\,\mbox{as }r\rightarrow\infty.\]
In $MR<x_{1}<MR+1$, $G(x,t)\leq e^{-\frac{\lambda\left|x'\right|^{2}}{8}+2^{\alpha}x_{1}^{2}+2b}$.
Hence, because of the bound (\ref{eq: thm 2 upper}) 

\[
\int_{0}^{T}\int_{MR<x_{1}<MR+1}\left(\left|u\right|^{2}+\left|\nabla u\right|^{2}\right)Gdxdt\lesssim1.\]
Furthermore,\[
\int_{0}^{T}\int_{R-1<x_{1}<R}\left(\left|u\right|^{2}+\left|\nabla u\right|^{2}\right)Gdxdt\lesssim e^{2^{\alpha}MR^{2}}.\]
Thus, we conclude that \begin{eqnarray*}
\int_{0}^{T}\int_{\mathbb{R}_{+}^{n}}\left(\left|v\right|^{2}+\left|\nabla v\right|^{2}\right)Gdxdt & \lesssim & e^{2^{\alpha}MR^{2}}.\end{eqnarray*}
As in $\{x:dR<x_{1}<2dR,\left|x'\right|<R\}$, $u=v$ and $G(x,t)\geq e^{(2^{\alpha}+1)MR^{2}}$,
it follows that\[
\int_{0}^{T/2}\int_{dR<x_{1}<2dR,\left|x'\right|<R}\left(\left|u\right|^{2}+\left|\nabla u\right|^{2}\right)dxdt\lesssim e^{-MR^{2}}.\]

\end{proof}
\begin{rem*}
Using the inequality (\ref{eq:doubling}) and a chain-of-balls argument,
we can actually take $d$ to be any positive number.
\end{rem*}

\subsection{Lower bound}

Assuming $u(\cdot,0)\not\equiv0$, the same argument as in section
2.2 gives the lower bound\[
\int_{B(Re_{1},\rho\lambda R)}\left|u(x,t)\right|^{2}dxdt\gtrsim e^{-R^{s}}\,\,\,\,\,\forall t\in[0,T],\,\,\,\forall R\gtrsim1\]
for some $T\leq1$.

For the second step, we will need another Carleman inequality. Let
$\delta\in(1,1+\epsilon/2)$ where $\epsilon$ is the constant in
the hypothesis of Theorem 1.2, and \[
S_{R,T}:=\left\{ (x,t)\in\mathbb{R}_{+}^{n}:R^{1/\delta}\leq x_{1}\leq R,T/8\leq t\leq7T/8\right\} .\]

\begin{lem}
Let $G(x,t)=e^{\varphi(x,t)}$ where \[
\varphi(x,t)=-\frac{\lambda\left|x'\right|^{2}}{8t}+E_{1}R\frac{(T^{\alpha}-t^{\alpha})}{t^{\alpha}}x_{1}+E_{2}\left(x_{1}-R\psi(t)\right)^{2}+bE_{2}t\]
where $\alpha$ and $b$ are suitable fixed positive constants, $E_{1},E_{2}\gtrsim1$
are large constants that may depend on $R$, but $E_{1}/E_{2}$ is
a large fixed constant independent of $R$. Then for large $R$, \begin{eqnarray*}
 &  & \int_{\mathbb{R}_{+}^{n+1}}v^{2}Gdxdt+\int_{\mathbb{R}_{+}^{n+1}}\left|\nabla v\right|^{2}Gdxdt\leq\int_{\mathbb{R}_{+}^{n+1}}\left|Pv\right|^{2}Gdxdt,\end{eqnarray*}
for any $v\in C_{c}^{\infty}(S_{R,T})$. Here and $\psi$ is as in
section 2.2. 
\end{lem}
We will omit the proof of this lemma as it is almost the same as that
of Lemma 4.1, except for the important fact that $E_{1}R\frac{(T^{\alpha}-t^{\alpha})}{t^{\alpha}}x_{1}$
is now the dominating term. (This is similar to the relationship between
Lemma \ref{lem:upper 1} and Lemma \ref{lem:lower 1})

\begin{prop}
\label{pro:iteration 2}Let $u$ and $P$ be as in Theorem 1.2. Suppose
that for some $s\geq2$, there are constants $R_{s},C_{s}>0$ such
that\begin{eqnarray*}
\int_{T/4}^{3T/4}\int_{\begin{array}{c}
R\leq x_{1}\leq2R\\
\left|x'\right|\leq C_{s}R^{s/2}\end{array}}(u^{2}+\left|\nabla u\right|^{2})dxdt & \gtrsim & \exp(-C_{s}R^{s})\end{eqnarray*}
 for all $R\geq R_{s}$. Let $s_{1}=\max\left\{ 2,\frac{s-1}{\delta}+1\right\} $
for some $1<\delta<1+\frac{\epsilon}{2}$. Then there is $R_{s_{1}},C_{s_{1}}$
such that 

\[
\int_{0}^{T}\int_{\begin{array}{c}
R\leq x_{1}\leq2R\\
\left|x'\right|\leq C_{s_{1}}R^{s_{1}/2}\end{array}}(u^{2}+\left|\nabla u\right|^{2})dxdt\gtrsim\exp(-C_{s_{1}}R^{s_{1}})\]
for all $R\geq R_{s_{1}}$.
\end{prop}
\begin{proof}
Let \[
v(x,t)=u(x,t)\theta(x,t)\mbox{ where }\theta(x,t)=\theta_{1}(x)\theta_{2}(x_{1}-R\psi(t))\theta_{3}(x'),\]
 where $\psi$ is defined as before, and \[
\theta_{1}(x)=\begin{cases}
\begin{array}{cc}
0 & \mbox{ if }x_{1}<R^{1/\delta}\mbox{ or }x_{1}>cR\\
1 & \mbox{ if }R^{1/\delta}+1\leq x_{1}\leq cR-1\end{array}\end{cases}\]
\[
\theta_{2}(r)=\begin{cases}
\begin{array}{cc}
0 & \mbox{ if }r<2R\\
1 & \mbox{ if }r>3R,\end{array}\end{cases}\]

\[
\theta_{3}(x')=\begin{cases}
\begin{array}{cc}
0 & \mbox{ if }\left|x'\right|>C_{s_{1}}R^{s_{1}/2}+1\\
1 & \mbox{ if }\left|x'\right|<C_{s_{1}}R^{s_{1}/2},\end{array}\end{cases}\]
where $c$ and $C_{s_{1}}$ are positive constants to be choosen.
It is clear that $\mbox{supp}(v)\subset S_{R,T}$.

Applying the previous Carleman inequality to $v$, as before we get\[
\int_{\mathbb{R}_{+}^{n+1}}(v^{2}+\left|\nabla v\right|^{2})Gdxdt\lesssim\int_{E}\left(\left|u\right|^{2}+\left|\nabla u\right|^{2}\right)Gdxdt,\]
where $E=\mbox{supp}\nabla\theta$. 

Because in the set $\{x:10^{\alpha}R^{1/\delta}\leq x_{1}\leq2\cdot10^{\alpha}R^{1/\delta},\left|x'\right|\leq C_{s}(10^{\alpha}R^{1/\delta})^{s/2}\}\times[\frac{T}{4},\frac{3T}{4}]$,\begin{eqnarray*}
G(x,t) & \geq & \exp\left(-D_{s}^{'}R^{s/\delta}+10^{\alpha}E_{1}R^{1+\frac{1}{\delta}}+E_{2}\left(4R-2\cdot10^{\alpha}R^{1/\delta}\right)^{2}\right)\end{eqnarray*}
we have, \begin{eqnarray*}
\int_{\mathbb{R}_{+}^{n+1}}(\left|v\right|^{2}+\left|\nabla v\right|^{2})Gdxdt & \gtrsim & \exp\left(10^{\alpha}E_{1}R^{1+\frac{1}{\delta}}+E_{2}\left(4R-2\cdot10^{\alpha}R^{1/\delta}\right)^{2}-C_{s}R^{s/\delta}-D_{s}^{'}R^{s/\delta}\right)\\
 & \gtrsim & \exp\left(9^{\alpha}E_{1}R^{1+\frac{1}{\delta}}+16E_{2}R^{2}\right)=:\Sigma\end{eqnarray*}
if $E_{1}/R^{\frac{s-1}{\delta}-1}$ and $E_{1}/E_{2}$ are large
enough.

In $\{R^{1/\delta}\leq x_{1}\leq R^{1/\delta}+1\}$, \[
G(x,t)\leq\exp\left(-\frac{\lambda\left|x'\right|^{2}}{8}+8^{\alpha}E_{1}R^{1+\frac{1}{\delta}}+16E_{2}R^{2}\right)\]
so using the bound (\ref{eq: thm 2 upper}) we get \[
\int_{0}^{T}\int_{R^{1/\delta}\leq x_{1}\leq2R^{1/\delta}}\left(\left|u\right|^{2}+\left|\nabla u\right|^{2}\right)Gdxdt\lesssim\exp\left(8^{\alpha}E_{1}R^{1+\frac{1}{\delta}}+16E_{2}R^{2}\right)\ll\Sigma.\]
In $\{2R\leq\left|x_{1}-R\psi(t)\right|\leq3R\}\cap\{x_{1}\leq cR\}$,
\[
G(x,t)\leq\exp\left(-\frac{\lambda\left|x'\right|^{2}}{8}+c8^{\alpha}E_{1}R^{2}+9E_{2}R^{2}\right)\]
Hence, if $c$ is chosen to be small enough, \[
\int_{0}^{T}\int_{2R\leq\left|x_{1}-R\psi(t)\right|\leq3R,x_{1}\leq cR}\left(\left|u\right|^{2}+\left|\nabla u\right|^{2}\right)Gdxdt\lesssim\exp\left(10E_{2}R^{2}\right)\ll\Sigma.\]
In $\{C_{s_{1}}R^{s_{1}/2}\leq\left|x'\right|\leq C_{s_{1}}R^{s_{1}/2}+1\}$,
\[
G(x,t)\leq\exp(-\lambda C_{s_{1}}^{2}R^{s_{1}}/8+c8^{\alpha}E_{1}R^{2}+16E_{2}R^{2})\]
Note that by our choice of $E_{1}$ and $E_{2}$, $E_{1}R^{2}\sim E_{2}R^{2}\sim R^{s_{1}}$,
so if we choose $C_{s_{1}}$ big enough, \[
\int_{0}^{T}\int_{x_{1}<cR,C_{s_{1}}R^{s_{1}/2}\leq\left|x'\right|\leq C_{s_{1}}R^{s_{1}/2}+1}\left(\left|u\right|^{2}+\left|\nabla u\right|^{2}\right)Gdxdt\lesssim1.\]
Thus, we conclude that \[
\Sigma\lesssim\int_{0}^{T}\int_{cR<x_{1}<cR+1,\left|x'\right|\leq C_{s_{1}}R^{s_{1}/2}}\left(\left|u\right|^{2}+\left|\nabla u\right|^{2}\right)Gdxdt.\]
Since in $\{cR<x_{1}<cR+1,\left|x'\right|\leq C_{s_{1}}R^{s_{1}/2}\}$,
$G\leq\exp(KR^{s_{1}})$, we obtain \[
\int_{0}^{T}\int_{\begin{array}{c}
cR\leq x_{1}\leq cR+1\\
\left|x'\right|\leq C_{s_{1}}R^{s_{1}/2}\end{array}}(\left|u\right|^{2}+\left|\nabla u\right|^{2})dxdt\geq\exp(-KR^{s_{1}}).\]
The proposition follows immediately from this.
\end{proof}
\begin{prop}
\label{pro:lower 1.2}Let $u$ and $P$ be as in Theorem 1.2. If $u(\cdot,0)\not\equiv0$
then then for any $T\leq1$, there exist $C_{2}=C_{2}(T,u)>0$ such
that 

\[
\int_{0}^{T}\int_{\begin{array}{c}
R\leq x_{1}\leq2R\\
\left|x'\right|\leq C_{2}R\end{array}}(u^{2}+\left|\nabla u\right|^{2})dxdt\gtrsim\exp(-C_{2}R^{2})\]
for all $R\geq C_{2}$.
\end{prop}
\begin{proof}
The proof is similar to that of Proposition \ref{pro:lower}, using
Proposition \ref{pro:iteration 2} instead of Proposition  \ref{pro:iteration}.
We omit the details.
\end{proof}
Using Proposition \ref{pro:upper 1.2} (see also the remark after
it) and Proposition \ref{pro:lower 1.2}, the proof of Theorem 1.2
is identical to that of Theorem 1.1. We omit the details.

\section{appendix}

\subsection{Some auxiliary lemmas}

The first lemma is a standard estimate for solutions of parabolic
inequalities, we refer to \cite{MR001}. 

\begin{lem}
\label{basic L8 bound}Suppose that in $\Omega^{*}:=B(0,2)\times[0,2R^{-2}]$
, the following inequality holds \[
\left|Pv\right|\leq R^{2}\left|v\right|+R\left|\nabla v\right|.\]
Then \[
\left\Vert v\right\Vert _{L^{\infty}(\Omega)}+\left\Vert \nabla v\right\Vert _{L^{\infty}(\Omega)}\leq C_{n}R^{c}\left\Vert v\right\Vert _{L^{2}(\Omega^{*})}\]
 where $\Omega=B(0,1)\times[0,R^{-2}]$ and $c$ is a constant depending
only on $n$. 
\end{lem}
The next two lemmas are from \cite{MR2198840} (see also \cite{MR2231129}). 

\begin{lem}
\label{Lemma}For $\rho\in(0,1/2)$, there is constant $N_{\rho}>0$
such that if \[
\left|Pv\right|\leq R^{2}\left|v\right|+R\left|\nabla v\right|\]
 in $\Omega^{*}:=B(0,2)\times[0,2R^{-2}]$ then \[
\int_{B(0,\rho)}v^{2}(x,0)dx\leq N_{\rho}\int_{B(0,2\rho)}v^{2}(x,t)dx\]
for all \[
0\leq t\leq N_{\rho}^{-1}\min\left\{ R^{-2},1/\log\left(\frac{N_{\rho}\int_{\Omega^{*}}v^{2}(x,0)dxdt}{\int_{B(0,\rho)}v^{2}(x,0)dx}\right)\right\} .\]

\end{lem}
~

\begin{lem}
\label{Lmm doubl}Suppose $v\in C_{c}^{\infty}(\mathbb{R}^{n})$ such
that for some $C>1$, \[
2a\int_{\mathbb{R}^{n}}\left|\nabla v\right|^{2}e^{-\left|x\right|^{2}/4a}dx+\frac{n}{2}\int_{\mathbb{R}^{n}}v^{2}e^{-\left|x\right|^{2}/4a}dx\leq C\int_{\mathbb{R}^{n}}v^{2}e^{-\left|x\right|^{2}/4a}dx,\]
for all $0<a\leq1/(12C)$. Then\[
\int_{B(0,2r)}v^{2}dx\leq e^{32C}\int_{B(0,r)}v^{2}dx\]
for all $0\leq r\leq1/2$.
\end{lem}

\subsection{Proof of the Carleman inequalities}

In this section we will prove the Carleman inequalities that were
used in the proofs of Theorems 1.1 and 1.2. We will use the following
notations\begin{eqnarray*}
 & \Delta v=\mbox{div}(A\nabla v)\\
 & \left\Vert \nabla v(x,t)\right\Vert =\left\langle A(x,t)\nabla v(x,t),\nabla v(x,t)\right\rangle ^{1/2}.\end{eqnarray*}
We recall the following lemma of \cite{MR1971939} (see also \cite{MR2001174},
\cite{MR2198840}). 

\begin{lem}
\label{Basic lmm}Suppose $\sigma(t):\mathbb{R}_{+}\rightarrow\mathbb{R}_{+}$
is a smooth function, $\alpha$ is a real number, $F$ and $G$ are
differentible functions, $G$ positive. Then the following identity
holds for $v\in C_{c}^{2}(\mathbb{R}^{n}\times[0,T])$\[
2\int_{\mathbb{R}_{+}^{n+1}}\frac{\sigma^{1-\alpha}}{\sigma'}w^{2}Gdxdt+\frac{1}{2}\int_{\mathbb{R}_{+}^{n+1}}\frac{\sigma^{1-\alpha}}{\sigma'}v^{2}MGdxdt-\frac{\alpha}{2}\int_{\mathbb{R}_{+}^{n+1}}\sigma^{-\alpha}v^{2}\left(\frac{\partial_{t}G-\Delta G}{G}-F\right)Gdxdt\]
\[
+\int_{\mathbb{R}_{+}^{n+1}}\frac{\sigma^{1-\alpha}}{\sigma'}\left[\left(\log\frac{\sigma}{\sigma'}\right)'+\frac{\partial_{t}G-\Delta G}{G}-F\right]\left\Vert \nabla v\right\Vert ^{2}Gdxdt+2\int_{\mathbb{R}_{+}^{n+1}}\frac{\sigma^{1-\alpha}}{\sigma'}\left\langle D_{G}\nabla v,\nabla v\right\rangle Gdxdt\]
\[
-\int_{\mathbb{R}_{+}^{n+1}}\frac{\sigma^{1-\alpha}}{\sigma'}v\left\langle A\nabla v,\nabla F\right\rangle Gdxdt=2\int_{\mathbb{R}_{+}^{n+1}}\frac{\sigma^{1-\alpha}}{\sigma'}wPvGdxdt+\int_{\mathbb{R}^{n}\times\{T\}}\frac{\sigma^{1-\alpha}}{\sigma'}\left\Vert \nabla v\right\Vert ^{2}Gdx-\]
\[
-\int_{\mathbb{R}^{n}\times\{0\}}\frac{\sigma^{1-\alpha}}{\sigma'}\left\Vert \nabla v\right\Vert ^{2}Gdx+\frac{1}{2}\int_{\mathbb{R}^{n}\times\{T\}}\frac{\sigma^{1-\alpha}}{\sigma'}v^{2}(F-\frac{\alpha\sigma'}{\sigma})Gdx-\]
\[
-\frac{1}{2}\int_{\mathbb{R}^{n}\times\{0\}}\frac{\sigma^{1-\alpha}}{\sigma'}v^{2}(F-\frac{\alpha\sigma'}{\sigma})Gdx.\]
where\[
w=\partial_{t}v-\left\langle A\nabla\log G,\nabla v\right\rangle +\frac{Fv}{2}-\frac{\alpha\sigma'}{2\sigma}v,\]
 \[
M=\left(\log\frac{\sigma}{\sigma'}\right)'F+\partial_{t}F+F\left(\frac{\partial_{t}G-\Delta G}{G}-F\right)-\left\langle A\nabla\log G,\nabla F\right\rangle ,\]

and \[
D_{G}^{ij}=a^{il}\partial_{kl}(\log G)a^{kj}+\frac{\partial_{l}(\log G)}{2}\left[a^{kj}\partial_{k}a^{il}+a^{ki}\partial_{k}a^{jl}-a^{kl}\partial_{k}a^{ij}\right]+\frac{1}{2}\partial_{t}a^{ij}.\]

\end{lem}
We will first derive a corollary of this lemma which will be used
to prove all of our Carleman inequalities. Letting $\alpha=0$ and
$\sigma(t)=e^{t}$ in Lemma \ref{Basic lmm}, we obtain the following
identity for $v\in C_{c}^{2}(\mathbb{R}^{n}\times[0,T])$,\begin{eqnarray}
 &  & 2\int_{\mathbb{R}_{+}^{n+1}}w^{2}Gdxdt+\frac{1}{2}\int_{\mathbb{R}_{+}^{n+1}}v^{2}MGdxdt-\int_{\mathbb{R}_{+}^{n+1}}v\left\langle A\nabla v,\nabla F\right\rangle Gdxdt\nonumber \\
 &  & +\int_{\mathbb{R}_{+}^{n+1}}\left\Vert \nabla v\right\Vert ^{2}\left(\frac{\partial_{t}G-\Delta G}{G}-F\right)Gdxdt+2\int_{\mathbb{R}_{+}^{n+1}}\left\langle D_{G}\nabla v,\nabla v\right\rangle Gdxdt\nonumber \\
 & = & 2\int_{\mathbb{R}_{+}^{n+1}}wPvGdxdt+\int_{\mathbb{R}^{n}}\left\Vert \nabla v(x,T)\right\Vert ^{2}Gdx-\int_{\mathbb{R}^{n}}\left\Vert \nabla v(x,0)\right\Vert ^{2}Gdx\label{eq: iden1}\\
 &  & +\frac{1}{2}\int_{\mathbb{R}^{n}}v^{2}(x,T)FGdx-\frac{1}{2}\int_{\mathbb{R}^{n}}v^{2}(x,0)FGdx.\nonumber \end{eqnarray}
where \[
M=\partial_{t}F+F\left(\frac{\partial_{t}G-\Delta G}{G}-F\right)-\left\langle A\nabla F,\nabla\log G\right\rangle .\]
Note that if $\nabla F$ is differentiable, we can integrate by parts
to obtain \[
-\int_{\mathbb{R}_{+}^{n+1}}v\left\langle A\nabla v,\nabla F\right\rangle Gdxdt=\frac{1}{2}\int_{\mathbb{R}_{+}^{n+1}}v^{2}\Delta FGdxdt+\frac{1}{2}\int_{\mathbb{R}_{+}^{n+1}}v^{2}\left\langle A\nabla F,\nabla\log G\right\rangle Gdxdt.\]
Then this term can be combined with the second term of the left hand
side. However, in our applications, $\nabla F$ might not be differentiable,
so we approximate $F$ by some $C^{2}$ function $F_{0}$ and use
the above identity with $F_{0}$ in place of $F$. Then, using Cauchy-Schwarz,
we arrive at the following lemma.

\begin{lem}
\label{lem:basic 2}Suppose $v\in C_{c}^{2}(\mathbb{R}^{n}\times[0,T])$,
then\begin{eqnarray}
 &  & \frac{1}{2}\int_{\mathbb{R}_{+}^{n+1}}v^{2}M_{0}Gdxdt+\int_{\mathbb{R}_{+}^{n+1}}\left[2\left\langle D_{G}\nabla v,\nabla v\right\rangle +\left\Vert \nabla v\right\Vert ^{2}\left(\frac{\partial_{t}G-\Delta G}{G}-F\right)\right]Gdxdt\nonumber \\
 &  & -\int_{\mathbb{R}_{+}^{n+1}}v\left\langle A\nabla v,\nabla(F-F_{0})\right\rangle Gdxdt\leq\int_{\mathbb{R}_{+}^{n+1}}\left|Pv\right|^{2}Gdxdt+\int_{\mathbb{R}^{n}}\left\Vert \nabla v(x,T)\right\Vert ^{2}Gdx\nonumber \\
 &  & -\int_{\mathbb{R}^{n}}\left\Vert \nabla v(x,0)\right\Vert ^{2}Gdx+\frac{1}{2}\int_{\mathbb{R}^{n}}v^{2}(x,T)FGdx-\frac{1}{2}\int_{\mathbb{R}^{n}}v^{2}(x,0)FGdx.\label{eq:inequa 1}\end{eqnarray}
where \[
M_{0}=\partial_{t}F+F\left(\frac{\partial_{t}G-\Delta G}{G}-F\right)+\Delta F_{0}-\left\langle A\nabla(F-F_{0}),\nabla\log G\right\rangle \]
and \[
D_{G}^{ij}=a^{il}\partial_{kl}(\log G)a^{kj}+\frac{\partial_{l}(\log G)}{2}\left[a^{kj}\partial_{k}a^{il}+a^{ki}\partial_{k}a^{jl}-a^{kl}\partial_{k}a^{ij}\right]+\frac{1}{2}\partial_{t}a^{ij}.\]

\end{lem}
We will now prove our Carleman inequalities using Lemma \ref{lem:basic 2}.

\begin{proof}
[Proof of lemma 2.2]As $\mbox{supp}v\subset\{R\leq\left|x\right|\leq R^{1+\epsilon/8}\}\times[0,T]$,
we will assume that $R\leq\left|x\right|\leq R^{1+\epsilon/8}$ in
all the computation below. 

Since $\nabla^{2}\varphi\geq\mbox{Id}$, $\left|\nabla\log G\right|\leq R^{1+\epsilon/8},$
and $\left|\nabla a^{ij}(x,t)\right|\leq R^{-1-\epsilon}$, it follows
that $D_{G}\geq\frac{\lambda^{2}}{2}\mbox{Id}$ for large $R$. To
make the gradient term (i.e. the second term in (\ref{eq:inequa 1}))
positive, we will choose $F$ satisfying\begin{equation}
\left|\frac{\partial_{t}G-\Delta G}{G}-F\right|\leq\lambda^{4}/2,\label{eq:proof 1 eq 1}\end{equation}
so that \[
2\left\langle D_{G}\nabla v,\nabla v\right\rangle +\left\Vert \nabla v\right\Vert ^{2}\left(\frac{\partial_{t}G-\Delta G}{G}-F\right)\geq\frac{\lambda^{2}}{2}\left|\nabla v\right|^{2}.\]
Let $\varphi(x,t)=c(T-t)\left|x\right|+\left|x\right|^{2}$, then
\begin{eqnarray*}
\frac{\partial_{t}G-\Delta G}{G} & = & \partial_{t}\varphi-\Delta\varphi-a^{ij}\partial_{i}\varphi\partial_{j}\varphi\\
 & = & -c\left|x\right|-x_{j}\partial_{i}a^{ij}(x,t)\left(c(T-t)\left|x\right|^{-1}+2\right)-a^{ij}(x,t)x_{i}x_{j}\left(c(T-t)\left|x\right|^{-1}+2\right)^{2}\\
 &  & -a^{ij}(x,t)\left[\delta_{ij}\left(c(T-t)\left|x\right|^{-1}+2\right)-c(T-t)x_{i}x_{j}\left|x\right|^{-3}\right],\end{eqnarray*}
As the second term is of order $O(R^{-7\epsilon/8})$, if we let \begin{eqnarray*}
F(x,t) & = & -c\left|x\right|+\frac{\lambda^{4}}{3}-a^{ij}(x,t)x_{i}x_{j}\left(c(T-t)\left|x\right|^{-1}+2\right)^{2}\\
 &  & -a^{ij}(x,t)\left[\delta_{ij}\left(c(T-t)\left|x\right|^{-1}+2\right)-c(T-t)x_{i}x_{j}\left|x\right|^{-3}\right].\end{eqnarray*}
then (\ref{eq:proof 1 eq 1}) is satisfied. Moreover, \[
-R^{2+\epsilon/4}\lesssim F\lesssim-R^{2},\,\,\,\,-\frac{\lambda^{4}}{2}\leq\frac{\partial_{t}G-\Delta G}{G}-F\leq-\frac{\lambda^{4}}{4}.\]
We have\begin{eqnarray*}
\partial_{t}F(x,t) & = & -\partial_{t}a^{ij}(x,t)x_{i}x_{j}\left(c(T-t)\left|x\right|^{-1}+2\right)^{2}+2ca^{ij}(x,t)x_{i}x_{j}\left|x\right|^{-1}\left(c(T-t)\left|x\right|^{-1}+2\right)\\
 &  & -\partial_{t}\left\{ a^{ij}(x,t)\left[\delta_{ij}\left(c(T-t)\left|x\right|^{-1}+2\right)-c(T-t)x_{i}x_{j}\left|x\right|^{-3}\right]\right\} .\end{eqnarray*}
The second term on the right hand side is positive by ellipticity
of $\{a^{ij}\}$. Noting that the last terms of $F$ and $\partial_{t}F$
are $O(R^{\epsilon/8})$, we get \begin{eqnarray*}
\partial_{t}F+F\left(\frac{\partial_{t}G-\Delta G}{G}-F\right) & \geq & -\left(\frac{\partial_{t}G-\Delta G}{G}-F\right)a^{ij}(x,t)x_{i}x_{j}\left(c(T-t)\left|x\right|^{-1}+2\right)^{2}\\
 &  & -\partial_{t}a^{ij}(x,t)a^{ij}(x,t)x_{i}x_{j}\left(c(T-t)\left|x\right|^{-1}+2\right)^{2}+O(R^{\epsilon/8})\\
 & \gtrsim & R^{2}.\end{eqnarray*}
(note that $\left|\partial_{t}a^{ij}(x,t)\right|\leq C\leq\lambda^{5}/100$.)

For the approximation $F_{0}$ of $F$, we choose \begin{eqnarray*}
F_{0}(x,t) & = & -c\left|x\right|+\frac{\lambda^{4}}{3}-a^{ij}(X,t)x_{i}x_{j}\left(c(T-t)\left|x\right|^{-1}+2\right)^{2}\\
 &  & -a^{ij}(X,t)\left[\delta_{ij}\left(c(T-t)\left|x\right|^{-1}+2\right)-c(T-t)x_{i}x_{j}\left|x\right|^{-3}\right],\end{eqnarray*}
where $X=(2R,0,\ldots,0)$. 

As \[
\left|a^{ij}(x,t)-a^{ij}(X,t)\right|=O(R^{-7\epsilon/8}),\,\,\,\left|\nabla a^{ij}(x,t)\right|=O(R^{-1-\epsilon})\]
and \begin{eqnarray*}
 &  & x_{i}x_{j}\left(c(T-t)\left|x\right|^{-1}+2\right)^{2}=O(R^{2+\epsilon/4}),\,\,\,\,\,\nabla\left(x_{i}x_{j}\left(c(T-t)\left|x\right|^{-1}+2\right)^{2}\right)=O(R^{1+\epsilon/4})\\
 &  & \delta_{ij}\left(c(T-t)\left|x\right|^{-1}+2\right)-c(T-t)x_{i}x_{j}\left|x\right|^{-3}=O(R^{\epsilon/8}),\,\,\,\,\,\\
 &  & \nabla\left(\delta_{ij}\left(c(T-t)\left|x\right|^{-1}+2\right)-c(T-t)x_{i}x_{j}\left|x\right|^{-3}\right)=O(R^{-1+\epsilon/8}),\end{eqnarray*}
we have\[
\nabla(F-F_{0})=O(R^{1-5\epsilon/8}).\]
Easy computation shows \[
\Delta F_{0}=O(R^{\epsilon/4}).\]
Thus, \[
M_{0}=\partial_{t}F+F\left(\frac{\partial_{t}G-\Delta G}{G}-F\right)+\Delta F_{0}-\left\langle A\nabla(F-F_{0}),\nabla\log G\right\rangle \gtrsim R^{2}.\]
 Finally, we can use Cauchy-Schwarz to control the remaining term
as follows \[
\left|\int_{0}^{T}\int_{\mathbb{R}^{n}}v\left\langle A\nabla v,\nabla(F-F_{0})\right\rangle Gdxdt\right|\leq\frac{M_{0}}{4}\int_{\mathbb{R}^{n}}v^{2}Gdxdt+\frac{\lambda^{2}}{4}\int_{\mathbb{R}^{n}}\left|\nabla v\right|^{2}Gdxdt.\]
This show that the left hand side of (\ref{lem:basic 2}) is greater
than \[
\frac{R^{2}}{4}\int_{\mathbb{R}^{n}}v^{2}Gdxdt+\frac{\lambda^{2}}{4}\int_{\mathbb{R}^{n}}\left|\nabla v\right|^{2}Gdxdt.\]
In our case, $F<0$ so the third and fourth terms in the right hand
side of (\ref{lem:basic 2}) are negative. Thus, the lemma is proved.
\end{proof}
~

\begin{proof}
[Proof of lemma 2.4]As $\nabla^{2}\varphi\geq2E_{2}\mbox{Id}$, the
first term in $D_{G}$ is at least $2\lambda^{2}E_{2}\mbox{Id}$.
The middle three terms of $D_{G}$ are $O(E_{1}R^{1-\frac{1+\epsilon}{\delta}})$,
and the last term is bounded by $C\leq\lambda^{4}$. Thus, $D_{G}\geq\lambda^{2}E_{2}\mbox{Id}$.
To make the gradient term positive, we will chose $F$ satisfying\[
\left|\frac{\partial_{t}G-\Delta G}{G}-F\right|\leq\lambda^{4}E_{2},\]
so that then \[
2\left\langle D_{G}\nabla v,\nabla v\right\rangle +\left\Vert \nabla v\right\Vert ^{2}\left(\frac{\partial_{t}G-\Delta G}{G}-F\right)\geq\lambda^{2}E_{2}\left|\nabla v\right|^{2}.\]
Let $\widetilde{x}=x-R\psi(t)e_{1}$. Then we have

\begin{eqnarray*}
\frac{\partial_{t}G-\Delta G}{G} & = & \partial_{t}\varphi-\Delta\varphi-a^{ij}\partial_{i}\varphi\partial_{j}\varphi\\
 & = & -E_{1}R\left|x\right|-2E_{2}R\psi'(t)\left(x_{1}-R\psi(t)\right)\\
 &  & -\partial_{i}a^{ij}(x,t)\left(E_{1}R(T-t)\frac{x_{j}}{\left|x\right|}+2E_{2}\widetilde{x_{j}}\right)\\
 &  & -a^{ij}(x,t)\left(E_{1}R(T-t)\frac{x_{i}}{\left|x\right|}+2E_{2}\widetilde{x_{i}}\right)\left(E_{1}R(T-t)\frac{x_{j}}{\left|x\right|}+2E_{2}\widetilde{x_{j}}\right)\\
 &  & -a^{ij}(x,t)\left[-E_{1}R(T-t)\frac{x_{i}x_{j}}{\left|x\right|^{3}}+\delta_{ij}\left(E_{1}R(T-t)\left|x\right|^{-1}+2E_{2}\right)\right].\end{eqnarray*}
Note that  in $S_{R,T}$ we have\[
\left|\nabla a^{ij}(x,t)\right|\leq\left\langle x\right\rangle ^{-1-\epsilon}\leq R^{-(1+\epsilon)/\delta},\]
hence\[
\left|\partial_{i}a^{ij}(x,t)\left(E_{1}R(T-t)\frac{x_{j}}{\left|x\right|}+2E_{2}\widetilde{x_{j}}\right)\right|\lesssim E_{1}R^{1-\frac{1+\epsilon}{\delta}}.\]
Thus, we choose\begin{eqnarray*}
F(x,t) & = & -E_{1}R\left|x\right|-2E_{2}R\psi'(t)\left(x_{1}-R\psi(t)\right)+\lambda^{4}E_{2}/2\\
 &  & -a^{ij}(x,t)\left(E_{1}R(T-t)\frac{x_{j}}{\left|x\right|}+2E_{2}\widetilde{x_{j}}\right)\left(E_{1}R(T-t)\frac{x_{i}}{\left|x\right|}+2E_{2}\widetilde{x_{i}}\right)\\
 &  & -a^{ij}(x,t)\left[-E_{1}R(T-t)\frac{x_{i}x_{j}}{\left|x\right|^{3}}+\delta_{ij}\left(E_{1}R(T-t)\left|x\right|^{-1}+2E_{2}\right)\right]\end{eqnarray*}
Also, let

\begin{eqnarray*}
F_{0}(x,t) & = & -E_{1}R\left|x\right|-2E_{2}R\psi'(t)\left(x_{1}-R\psi(t)\right)+\lambda^{4}E_{2}/2\\
 &  & -a^{ij}(X,t)\left(E_{1}R(T-t)\frac{x_{j}}{\left|x\right|}+2E_{2}\widetilde{x_{j}}\right)\left(E_{1}R(T-t)\frac{x_{i}}{\left|x\right|}+2E_{2}\widetilde{x_{i}}\right)\\
 &  & -a^{ij}(X,t)\left[-E_{1}R(T-t)\frac{x_{i}x_{j}}{\left|x\right|^{3}}+\delta_{ij}\left(E_{1}R(T-t)\left|x\right|^{-1}+2E_{2}\right)\right]\end{eqnarray*}
where $X=(2R^{1/\delta},0,\ldots,0)$.

In the support of $v$, $T-t\geq T/8$, and $\left|\tilde{x}\right|\leq5R$,
so by ellipticity of $\left\{ a^{ij}\right\} $, \begin{eqnarray*}
 &  & -a^{ij}\left(E_{1}R(T-t)\frac{x_{j}}{\left|x\right|}+2E_{2}\widetilde{x_{j}}\right)\left(E_{1}R(T-t)\frac{x_{i}}{\left|x\right|}+2E_{2}\widetilde{x_{i}}\right)\\
 &  & \leq-\lambda\left|E_{1}R(T-t)\frac{x}{\left|x\right|}-2E_{2}\tilde{x}\right|^{2}\lesssim-T^{2}E_{1}^{2}R^{2}.\end{eqnarray*}
The other terms in $F$ are bounded by $E_{1}R^{2}$, $E_{2}R^{2}/T$,
$E_{2}$, and $E_{1}TR^{1-\frac{1}{\delta}}$. Hence, for large $R$,
\[
F\lesssim-E_{1}^{2}T^{2}R^{2}\mbox{ and }F\left(\frac{\partial_{t}G-\Delta G}{G}-F\right)\gtrsim E_{1}^{2}E_{2}R^{2}\]
It is easy to check that \begin{eqnarray*}
\partial_{t}F & = & O(E_{1}^{2}R^{2})\\
\Delta F_{0} & = & O(E_{1}^{2}R^{2-\frac{2}{\delta}})\end{eqnarray*}
which is smaller than $F\left(\frac{\partial_{t}G-\Delta G}{G}-F\right)$
provided $E_{2}\gg1$. Using $\left|a^{ij}(x,t)-a^{ij}(X,t)\right|\lesssim R^{1-\frac{1+\epsilon}{\delta}}$
and $\left|\nabla\left(a^{ij}(x,t)-a^{ij}(X,t)\right)\right|=R^{-\frac{1+\epsilon}{\delta}}$
in $S_{R,T}$, we get \[
\left|\nabla(F-F_{0})\right|\lesssim R^{3-\frac{2+\epsilon}{\delta}}E_{1}^{2}\]
hence \[
\left|\left\langle A\nabla(F-F_{0}),\nabla\log G\right\rangle \right|\lesssim R^{4-\frac{2+\epsilon}{\delta}}E_{1}^{3}\ll E_{1}^{2}E_{2}R^{2}\]
for large $R$, as $\delta<1+\frac{\epsilon}{2}$. 

Putting together these estimates, we obtain \[
M_{0}=\partial_{t}F+\Delta F_{0}+F\left(\frac{\partial_{t}G-\Delta G}{G}-F\right)-\left\langle A\nabla(F-F_{0}),\nabla\log G\right\rangle \geq E_{1}^{2}E_{2}R^{2}.\]
Finally, since \[
M_{0}E_{2}\gg R^{6-\frac{2(2+\epsilon)}{\delta}}E_{1}^{4}\gtrsim\left|\nabla(F-F_{0})\right|^{2},\]
we can control the remaining term by Cauchy-Schwarz, \[
\left|\int_{\mathbb{R}_{+}^{n+1}}u\left\langle A\nabla u,\nabla(F-F_{0})\right\rangle Gdxdt\right|\leq\frac{M_{0}}{2}\int u^{2}Gdxdt+\frac{\lambda^{2}E_{2}}{8}\int\left|\nabla u\right|^{2}Gdxdt\]
Thus, the lemma is proved.
\end{proof}
~

\begin{proof}
[Proof of lemma 3.1]As $\nabla^{2}\varphi\geq-\frac{\lambda}{8}\mbox{Id}$
and\[
\left|\nabla\log G\right|\left|\nabla a^{ij}\right|=O(\left|x\right|^{-7\epsilon/8}),\,\,\,\,\left|\partial_{t}a^{ij}\right|\leq C\leq\lambda^{4}/100,\]
it follows that if \[
H:=\frac{\partial_{t}G-\Delta G}{G}-F\geq\frac{1}{\lambda}\]
then the gradient term is positive. We have \begin{eqnarray*}
\frac{\partial_{t}G-\Delta G}{G} & = & \frac{\lambda}{16s^{2}}\sum_{i,j\ne1}\left(2\delta_{ij}-\lambda a^{ij}(x,t)\right)x_{i}x_{j}-\frac{c\alpha S^{\alpha}x_{1}}{s^{\alpha+1}}+b\\
 &  & -a^{11}(x,t)\frac{c^{2}(S^{\alpha}-s^{\alpha})^{2}}{s^{2\alpha}}-\sum_{j\ne1}a^{1j}(x,t)\frac{\lambda x_{j}}{2s}\frac{c(S^{\alpha}-s^{\alpha})}{s^{\alpha}}\\
 &  & -\partial_{i}(a^{ij}\partial_{j}\varphi).\end{eqnarray*}
Since $c\leq x_{1}^{1+\epsilon/8}$, from the decay of $\nabla a^{ij}$
it follows that $\left|\partial_{i}(a^{ij}\partial_{j}\varphi)\right|\lesssim1.$
If we choose\begin{eqnarray*}
F(x,t) & = & \frac{\lambda}{16s^{2}}\sum_{i,j\ne1}\left(2\delta_{ij}-\lambda a^{ij}(x,t)\right)x_{i}x_{j}-\frac{c\alpha S^{\alpha}x_{1}}{s^{\alpha+1}}\\
 &  & -a^{11}(x,t)\frac{c^{2}(S^{\alpha}-s^{\alpha})^{2}}{s^{2\alpha}}-\sum_{j\ne1}a^{1j}(x,t)\frac{\lambda x_{j}}{2s}\frac{c(S^{\alpha}-s^{\alpha})}{s^{\alpha}},\end{eqnarray*}
then for large $b$, \[
2b\geq H\geq b/2\]
implying the positivity of the gradient term.

Consider four terms of $(H+\partial_{t})F$ corresponding to four
terms of $F$. 
\begin{enumerate}
\item $I_{1}:=\left(\frac{1}{2}Hs-1\right)\frac{\lambda}{8s^{3}}\sum_{i,j\ne1}\left(2\delta_{ij}-\lambda a^{ij}(X,t)\right)x_{i}x_{j}-\sum_{i,j\ne1}\partial_{t}a^{ij}x_{i}x_{j}\geq b\left|x'\right|^{2}$
\\
for large $b$.
\item \[
I_{2}:=\left(\frac{\alpha+1}{s}-H\right)\frac{c\alpha S^{\alpha}x_{1}}{s^{\alpha+1}}\geq\frac{c\alpha^{2}S^{\alpha}x_{1}}{4s^{\alpha+1}}\geq\frac{1}{8}c\alpha^{2}x_{1}\]
if $\alpha\geq4b$.
\item \begin{eqnarray*}
I_{3}: & = & \left(\frac{2\alpha}{s}-H\right)\frac{a^{11}(x,t)c^{2}(S^{\alpha}-s^{\alpha})^{2}}{s^{2\alpha}}+\frac{2a^{11}\alpha(S^{\alpha}-s^{\alpha})}{s^{\alpha+1}}-\\
 &  & -\partial_{t}a^{11}(x,t)\frac{c^{2}(S^{\alpha}-s^{\alpha})^{2}}{s^{2\alpha}}\geq\frac{\lambda\alpha c^{2}(S^{\alpha}-s^{\alpha})^{2}}{s^{2\alpha}},\end{eqnarray*}
\\
again if $\alpha\geq4b$.
\item \begin{eqnarray*}
I_{4}: & = & -a^{1j}(x,t)\frac{\lambda cx_{j}}{2}\left(-\frac{(\alpha+1)S^{\alpha}}{s^{\alpha+2}}+\frac{1}{s^{2}}+\frac{H(S^{\alpha}-s^{\alpha})}{s^{\alpha+1}}\right)\\
 &  & -\partial_{t}a^{1j}(x,t)\frac{\lambda cx_{j}}{2}\frac{S^{\alpha}-s^{\alpha}}{s^{\alpha+1}}.\end{eqnarray*}
\\
Since we are assuming $a_{\infty}^{1j}=0$, $\left|a^{1j}(x,t)\right|\lesssim\left\langle x\right\rangle ^{-\epsilon}$,
hence \[
\left|I_{4}\right|\leq2^{\alpha}\alpha c\left|x'\right|\left\langle x\right\rangle ^{-\epsilon}+\left|\partial_{t}a^{1j}(x,t)\right|\frac{\lambda c\left|x'\right|}{2}\frac{S^{\alpha}-s^{\alpha}}{s^{\alpha+1}}\]
Recall that $c\leq R^{1+\epsilon/8}\leq x_{1}^{1+\epsilon/8}$, hence
the first term is bounded by $\frac{1}{4}(I_{1}+I_{2})$. Also, by
dilation, we can assume $\left|\partial_{t}a^{ij}\right|\leq C\ll1$,
so that the second term is bounded by $\frac{1}{4}(I_{1}+I_{3})$.
Thus, \[
\left|I_{4}\right|\leq\frac{I_{1}}{2}+\frac{I_{2}+I_{3}}{4}.\]

\end{enumerate}
From these estimates, we obtain \[
(H+\partial_{t})F\geq\frac{I_{1}+I_{2}}{2}\geq\frac{1}{32}(b\left|x'\right|^{2}+c\alpha^{2}x_{1}).\]
As an approximation of $F$, we choose \begin{eqnarray*}
F_{0}(x,t) & = & \frac{\lambda}{16s^{2}}\sum_{i,j\ne1}\left(2\delta_{ij}-\lambda a^{ij}(X,t)\right)x_{i}x_{j}-\frac{c\alpha S^{\alpha}x_{1}}{s^{\alpha+1}}\\
 &  & -a^{11}(X,t)\frac{c^{2}(S^{\alpha}-s^{\alpha})^{2}}{s^{2\alpha}}-\sum_{j\ne1}a^{1j}(X,t)\frac{\lambda x_{j}}{2s}\frac{c(S^{\alpha}-s^{\alpha})}{s^{\alpha}},\end{eqnarray*}
where $X=(R,0,\ldots,0)$. Simple calculation shows that \[
\left|\Delta F_{0}\right|\lesssim\left\langle x\right\rangle ^{-1-\epsilon}(\left|x'\right|+c)+1\lesssim1\]
and 

\[
\left|\left\langle A\nabla(F-F_{0}),\nabla\log G\right\rangle \right|\lesssim\left\langle x\right\rangle ^{-1-\epsilon}\left(\left|x'\right|+c\right)^{3}+R^{-\epsilon}\left(\left|x'\right|+c\right)^{2}\lesssim R^{-\epsilon/2}(b\left|x'\right|^{2}+c\alpha^{2}x_{1}).\]
(the implicit constants depend on $\lambda$ but not on $R$). 

It follows that for large $R$,

\[
M_{0}=(\partial_{t}+H)F+\Delta F_{0}-\left\langle A\nabla(F-F_{0}),\nabla\log G\right\rangle \gtrsim b\left|x'\right|^{2}+c\alpha^{2}x_{1}.\]
We use Cauchy-Schwarz to control the remaining term \[
\left|\int_{\mathbb{R}_{+}^{n+1}}u\left\langle A\nabla u,\nabla(F-F_{0})\right\rangle Gdxdt\right|\leq\frac{1}{4}\int u^{2}M_{0}Gdxdt+\frac{b}{4}\int\left|\nabla u\right|^{2}Gdxdt.\]
As $\left|F\right|\lesssim\left|x'\right|^{2}+R^{2+\epsilon}$, the
lemma is proved.
\end{proof}

\title{\bibliographystyle{amsplain}
\bibliography{UC}
}

\noun{\footnotesize ~}{\footnotesize \par}

\noun{\footnotesize Department of Mathematics, University of Chicago,
5734 S. University Ave., Chicago, IL 60637, USA}{\footnotesize \par}

\textit{\footnotesize E-mail address}: \texttt{\footnotesize tu@math.uchicago.edu}
\end{document}